\documentclass[10pt,reqno]{amsart}

\usepackage{amsmath, amssymb, cases, color}
\usepackage{hyperref}
\usepackage{stmaryrd}%for double brackets notation!!!
\usepackage{amsmath, amssymb,amscd, }
\usepackage{amsfonts}
\usepackage{graphicx}

 \usepackage{upref}
\hypersetup{linkcolor=blue, colorlinks=true,citecolor = red}
\newtheorem{thm}{Theorem}[section]
\newtheorem{prop}[thm]{Proposition}

\newtheorem{cor}[thm]{Corollary}
\newtheorem{rmk}[thm]{Remark}

\theoremstyle{definition}
\newtheorem{definition}[thm]{Definition}

\theoremstyle{remark}
\numberwithin{equation}{section}

\newcommand{\Ran}{\operatorname{Ran}}

\newcommand{\BE}{\begin{equation}}
\newcommand{\EEQ}{\end{equation}}
\newcommand{\rfb}[1]{\mbox{\rm
		(\ref{#1})}\ifx\undefined\stillediting\else:\fbox{$#1$}\fi}

\newfont{\roma}{cmr10 scaled 1200}
\renewcommand{\cline}{{\mathbb C}}

\newcommand{\nline}  {{\mathbb N}}
\newcommand{\rline}  {{\mathbb R}}

\newcommand{\tline}  {{\mathbb T}}

\newcommand{\zline}  {{\mathbb Z}}

\newcommand{\dd}  {{\rm d}\hbox{\hskip 0.5pt}}
\renewcommand{\leq} {\leqslant}
\renewcommand{\geq} {\geqslant}
\newcommand{\mm}    {{\hbox{\hskip 0.5pt}}}
\newcommand{\m}     {{\hbox{\hskip 1pt}}}

\newcommand{\bluff} {{\hbox{\raise 15pt \hbox{\mm}}}}
\newcommand{\sbluff}{{\hbox{\raise 10pt \hbox{\mm}}}}
\newcommand{\Om}    {{\Omega}}

\newcommand{\FORALL} {{\hbox{$\hskip 11mm \forall \;$}}}
\newcommand{\prt}      {{\partial}}

\newcommand{\Dscr} {\mathcal{D}}
\newcommand{\Hscr} {\mathcal{H}}

\newcommand{\Lscr} {\mathcal{L}}

\newcommand{\Oscr} {\mathcal{O}}

\newcommand{\Sscr} {\mathcal{S}}

\begin{document}

\title[]{Shallow water waves generated by a floating object:
	    a control theoretical perspective}

 \date{\today}

 \author{Pei Su}
 \address{Institut de Math\'ematiques de Bordeaux,
 	Universit\'e de Bordeaux,
 	351, Cours de la Lib\'eration - F 33 405 TALENCE, France}
 \email[Corresponding author]{pei.su@u-bordeaux.fr}
 \urladdr{https://sites.google.com/view/peisu/accueil}
 
 \thanks{ The authors are members of the ETN network ConFlex, funded
 	by the European Union's Horizon 2020 research and innovation programme
 	under the Marie Sklodowska-Curie grant agreement no. 765579. The second 
 	author acknowledges the support of the SingFlows project, grant ANR-18-CE40-0027 
 	of the French National Research Agency (ANR)}

\author{Marius Tucsnak}
 \address{Institut de Math\'ematiques de Bordeaux,
	Universit\'e de Bordeaux,
	351, Cours de la Lib\'eration - F 33 405 TALENCE, France}
\email{marius.tucsnak@u-bordeaux.fr}    \urladdr{https://www.math.u-bordeaux.fr/~mtucsnak/}

\begin{abstract}
We consider a control system describing the interaction of water waves
with a partially immersed rigid body constraint to move only in the
vertical direction. The fluid is modeled by the shallow water equations.
The control signal is a vertical force acting on the floating body.
We first derive the full governing equations of the fluid-body system
in a water tank and reformulate them as an initial boundary value problem
of a first-order evolution system. We then linearize the equations around
the equilibrium state and we study its well-posedness. Finally we focus
on the reachability and stabilizability of the linear system. Our main
result asserts that, provided that the floating body is situated in
the middle of the tank, any symmetric waves with appropriate regularity
can be obtained from the equilibrium state by an appropriate control
force. This implies, in particular, that we can project this system on
the subspace of states with appropriate symmetry properties to obtain a 
reduced system which is approximately controllable and strongly stabilizable. 
Note that, in general, this system is not controllable (even approximately).
\end{abstract}
\maketitle

 {\bf Key words.} Shallow water equations, fluid-structure interactions,
 reachability, stabilizability, operator semigroup, infinite dimensional
 system.

{\bf 2020 AMS subject classifications.} 93B03, 93C20, 93D15, 76B15, 35Q35

%\tableofcontents

%%%%%%%%%%++++++++++%%%%%%%%%%++++++++++%%%%%%%%%%++++++++++%%%%%%%%%%++
\section{Introduction}\label{sec_intro}

In this work we are interested in the following problem: given a rigid
body floating in a fluid at rest in a bounded container, determine the
control force acting on the body in order to obtain a prescribed wave
profile. We assume that the floating object has vertical lateral walls,
with a possibly non-flat but symmetric bottom. More precisely, we assume
that the rigid body is restricted to the heave motion (move vertically)
and that it floats in a rectangular fluid domain which fits in the shallow
water regime (for this concept, please refer to Lannes \cite{lannes2013water,
lannes2020modeling} or Whitham \cite{whitham2011linear}). Moreover, the
body is actuated by a vertical control force and in the horizontal
direction it does not touch the lateral boundaries of the container.
The main contribution in this work consists in showing that, within the
linearized shallow water regime and in a spatially symmetric geometry,
we can find controls steering the system from rest to any symmetric wave
profile having an appropriate space regularity. In order to achieve this
goal we pass through the following preliminary steps:
\begin{itemize}
  \item Deriving the full nonlinear control model and reformulate it as a
   first-order evolution system;
  \item Establishing the well-posedness of the linearized control system.
\end{itemize}

The system we consider is also of interest for modelling and controlling
a class of {\em wave energy converters} (WECs) where all devices are used 
to capture the variations of the free surface waves and convert them into
electricity. The most popular WECs is the so-called {\em Point Absorber},
which consists of a floater on the sea surface and hydraulic cylinders
vertically installed below the floater (for more details, please refer to
Li et al. \cite{li2012wave} and Cretel et al. \cite{cretel2011maximisation}).
Mathematically speaking, this device acting from the bottom of the floating
body produces a vertical force, as a control signal, to synchronize the
motion of the body and of incoming waves and so maximize the energy
production or generate a desired waves.

There are a number of works which are devoted to the subject of
fluid-structure interaction systems. For instance, the case of the body
completely immersed in the fluid is studied in Glass et al.
\cite{glass2016motion}, Lacave and Takahashi \cite{lacave2017small} and
the corresponding control problem is considered in Roy and Takahashi
\cite{roy2020stabilization}, Glass et al. \cite{glass2020external}.
The case when the body is floating i.e. only partially immersed in the
fluid, is setup studied in John \cite{john1949motion, john1950motion} under
simplified assumptions. Recently, Lannes gave in \cite{lannes2017dynamics}
a new formulation of the governing equations and proposed a formulation of
the problem as a coupling between a standard wave model (in which the
surface elevation is free and the pressure is constrained) and a congested
model containing an object (where the pressure is free and the surface
elevation is constrained); this method can be implemented with various
asytmptotic models: non-viscous 1D shallow water model in Iguchi and Lannes
\cite{iguchi2021hyperbolic}, viscous 1D shallow water model in Maity et
al. \cite{maity2019analysis}, 2D radial symmetric shallow water equations
in Bocchi \cite{bocchi2020floating}, Boussinesq equations in Bresch et al.
\cite{bresch2019waves} and also in Beck and Lannes \cite{beck2021freely}.
We also refer to Godlewski et al. \cite{godlewski2018congested} where the
constraint for the equations with the object is released, using a typical
''low Mach'' technique. For other interesting formulations and asymptotic
models (depending on the shallowness parameter) for the water waves system
we refer to Lannes \cite{lannes2013water, lannes2020modeling} and references
therein. As far as we know, all the references on floating bodies mentioned
above are only concerned with the object freely floating  in the fluid and
there are almost no work on the control issue.

%%%%%%%%%%++++++++++%%%%%%%%%%++++++++++%%%%%%%%%%++++++++++%%%%%%%%%%
\subsection{Notation}\label{notation}
We introduce here, constantly referring to Figure \ref{fig1f}, some
notation which is used throughout this paper. We take the coordinate
system as in Figure \ref{fig1f}, where the ordinate axis passes through
the center of the floating object. The set $\mathcal{I}:=[-l,l]$, called the
%%%%%%%%%%++++++++++%%%%%%%%%%++++++++++%%%%%%%%%%++++++++++%%%%%%%%%%
\begin{figure}[htbp]
	\centering\includegraphics[width=8.6cm]{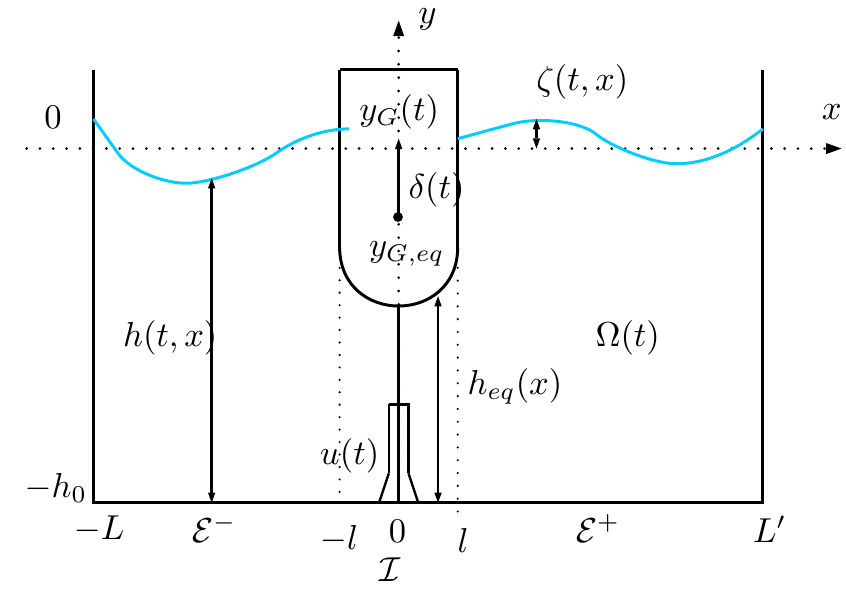}
	\caption{Floating body in a tank filled with water} \label{fig1f}
\end{figure} %\vspace{-2mm}
%%%%%%%%%%++++++++++%%%%%%%%%%++++++++++%%%%%%%%%%++++++++++%%%%%%%%%%
{\em interior region} in the remaining part of this work, is the
projection of the object on the bottom of the fluid domain $\Omega(t)$.
The {\em exterior region} is denoted by $\mathcal{E}:=\mathcal{E}^-\cup
\mathcal{E}^+$ with $\mathcal{E}^-=(-L,-l)$ and $\mathcal{E}^+=(l,L')$.
With the above notation, we assume that the object does not touch one
of the lateral boundaries of $\Om(t)$, i.e. $L\neq l$ and $L'\neq l$.

Let $h_0$ denote the water depth when the object is at equilibrium state.
In the same situation of equilibrium, let $(0, y_{G,\rm eq})$ denote the
coordinate of the center of gravity of the object and let $h_{\rm eq}(x)$
denote the distance between the point of abscissa $x$ of the bottom of
the object and the bottom of the fluid domain. We assume that the bottom
of the object is symmetric with respect to $x=0$, which implies that
$h_{\rm eq}(x)$ is a positive single-valued even function. We denote by $m$
the mass of the object, by $\rho$ the constant density of the fluid.
We also denote by $\zeta(t,x)$ the elevation of the water surface with
respect to the rest state, by $h(t,x)=h_0+\zeta(t,x)$ the total height
of the water column. Moreover, we introduce the horizontal discharge,
denoted by $q(t,x)$, that is the vertical integral of the horizontal
velocity of the fluid (in shallow water regime, it is $h$ times the
velocity of the fluid). We use the notation $\underline{P}$ to represent 
the pressure on the water surface. When the object moves in the vertical
direction, let $(0,y_{G}(t))$ be the position of the center of gravity 
at time $t$, and $\delta(t)=y_{G}(t)-y_{G,\rm eq}$ be the variation of 
the position of the center of mass. Furthermore, the vertical control force
acting on the object at time $t$ is denoted by $u(t)$.

We define the jump and  the average of a function $f$ defined on
$[-l,l]$ by $\llbracket f\m\rrbracket=f(l)-f(-l)$ and $\langle f
\m\rangle=\frac{1}{2}(f(l)+f(-l))$, respectively. Moreover,
$f_{\rm i}=f|_{\mathcal{I}}$ stands for the restriction of $f$ to
the interior domain $\mathcal{I}$ and $f_{\rm e}=f|_{\mathcal{E}}$
denotes the restriction of $f$ to the exterior domain $\mathcal{E}$.

%Throughout this paper, the notation $\nline,\ \zline,\ \rline,\ \cline$
%stands for the sets of natural numbers (starting with $1$), integers, real
%numbers and complex numbers, respectively. We denote $\zline^*=\zline
%\setminus\{0\}$.
If $k\in\nline$ and $\Oscr\subset\rline$ is an open set, we use the
notation $\Hscr^k(\Oscr)$ for the {\em Sobolev space} formed by the
distributions $f\in\Dscr'(\Oscr)$ having the property that
$\partial_x^\alpha f\in L^2(\Oscr)$ for every integer $\alpha\in[0,k] $.
%For $f\in\Hscr^k(\Oscr)$ we set
%$$ \|f\|_k^2 \m=\m \sum_{|\alpha|\leq k} \|\partial^\alpha f\|^2_{L^2}. $$
Finally, if a function $f$ depends on the time $t$, we denote by $\dot{f}$
its derivative with respect to $t$. For a matrix $M$, we denote by
$M^\intercal$ the transpose of $M$. We use the notation $X^\perp$ to
represent the orthogonal complement of the space $X$. For a complex number
$\alpha\in\cline$, we use $\overline{\alpha}$ to represent the complex
conjugate of $\alpha$.

%%%%%%%%%%++++++++++%%%%%%%%%%++++++++++%%%%%%%%%%++++++++++%%%%%%%%%%
\subsection{Main results}
The departure point of our derivation of the control system describing
the interaction of the floating body with the fluid is a nonlinear
model introduced in Lannes \cite{lannes2017dynamics}, where the fluid
fills an infinite strip in the horizontal direction. Taking the
control term into account, the governing equations of the floating body
system in the fluid domain $\Om(t)$ can be obtained from the conservation
laws of the total energy and of the volume of the water. In this case,
the interior surface pressure $\underline P_{\rm i}$ is not only determined 
by the fluid dynamics, but also by the external vertical force below the
floater. We show that $\underline P_{\rm i}$ satisfies a second-order
elliptic equation, and its source term and boundary term are given in 
terms of $\delta$, $\langle q_{\rm i}\rangle$ and the exterior functions
$\zeta_{\rm e}$, $q_{\rm e}$. Based on the nonlinear shallow water equations
and Newton's equation, we derive the equations for $\delta$ and
$\langle q_{\rm i}\rangle$ and find that their source terms again
consist of the exterior functions, respectively. In this way, the whole
system is converted to an initial and boundary value problem defined only 
in the exterior domain $\mathcal{E}$. Furthermore, it can be reformulated 
as a first-order evolution equation with the state $z$ as
\begin{equation*}\label{statez}
z=\begin{bmatrix}\m \zeta & q & \langle q_{\rm i}\rangle &
\delta & \dot \delta\m \end{bmatrix}^\intercal.
\end{equation*}
For the derivation of the fully nonlinear model, please refer to Section
\ref{sec_model}. Using the notations introduced above, we linearize the
nonlinear model around the equilibrium state  $\begin{bmatrix}\m \zeta &
q & \langle q_{\rm i}\rangle & \delta & \dot \delta\m \end{bmatrix}^\intercal
=\begin{bmatrix}\m 0 & 0 & 0 & 0 & 0\m \end{bmatrix}^\intercal$
and the resulting linearized fluid-body system, for every $t\geq 0$ and
$x\in\mathcal{E}$, reads
\begin{equation}\label{lineareqsintro}
\left\{\begin{aligned}
&\prt_t\zeta=-\prt_x q,\\
&\prt_t q=-gh_0\m\prt_x\zeta,\\
&\frac{\dd}{\dd t}\langle q_{\rm i}\rangle=-\frac{g}{2l\underline \alpha}
\llbracket \zeta\m\rrbracket,\\
&\ddot{\delta}=-\frac{2\rho g\m l}{\underline M}\delta+\frac{2\rho
	g\m l}{\underline M}\langle\zeta\m\rangle+\frac{1}{\underline M}u,
\end{aligned}\right.
\end{equation}
with the transmission conditions
\begin{equation}\label{trcon1}
 \langle q\rangle=\langle q_{\rm i}\rangle, \qquad \llbracket q
\rrbracket =-2l\dot{\delta},
\end{equation}
and boundary conditions
\begin{equation}\label{boun1}
q(t,-L)=0=q(t,L').
\end{equation}
The constants $\underline \alpha$ and $\underline M$ in
\rfb{lineareqsintro} are defined in Section \ref{sec_control}.
%$$\underline \alpha=\frac{1}{2l}\int_{-l}^l\frac{1}{h_{ eq}(x)}\dd x,
%\qquad \underline M=m+\int_{-l}^l\frac{\rho\m x^2}{h_{eq}(x)}\dd x. $$
Let the initial data of \rfb{lineareqsintro} be
\begin{equation}\label{init1}
z_0=\begin{bmatrix}
\m \zeta_0 & q_0 & \langle q_{\rm i}\rangle_0
& \delta_0 & \delta_1\m \end{bmatrix}^\intercal.
\end{equation}

Our first result is the well-posedness of the linear system 
\rfb{lineareqsintro}--\rfb{init1}. For the precise definition of the 
notion of solution of \rfb{lineareqsintro}--\rfb{init1}
we refer to Section \ref{sec_control}.

\begin{thm}\label{the1st}
The linearized floating body system \rfb{lineareqsintro}--\rfb{init1} forms
a linear control system with the state space
$$X=\left\{\begin{bmatrix}\m \zeta & q & \langle q_{\rm i}\rangle &
\delta & \eta\m \end{bmatrix}^\intercal\in \left(L^2(\mathcal{E})
\right)^2\times\cline^3\m\left|\m\int_{\mathcal{E}}\zeta(x)\dd x+2l
\m\delta=0\right. \right\}$$
and the input space $U=\cline$. For $u\in L^2_{\rm loc}([0,\infty); U))$,
the initial data $z_0\in X$, the system \rfb{lineareqsintro}--\rfb{init1}
admits a unique solution $z\in C([0,\infty); X)$.
\end{thm}

Our main interest is to study the reachable space of the control system
\rfb{lineareqsintro}--\rfb{init1}, when the object is put in the middle
of the fluid domain in the horizontal direction i.e. $L=L'$. This space
is formed of all the states that can be reached from equilibrium by means
of $L^2$ controls $u$. For every $\tau>0$, the bounded linear map $\Phi_\tau:
L^2([0,\infty);U)\to X$ is called an {\em input-to-state map} (briefly,
input map) of the system \rfb{lineareqsintro}--\rfb{init1} with zero
initial data (i.e. $z_0=0$) defined by
\begin{equation}\label{phitauintro}
\Phi_\tau u=z(\tau) \FORALL u\in L^2_{\rm loc}([0,\infty); U)).
\end{equation}
Notice that when $L'=L$ and the initial state is an equilibrium one,
the whole floating body-fluid system preserves its symmetry for all
$t\geqslant 0$, in the sense that $\zeta$ and $q$ satisfy
$$\zeta(t,-x)=\zeta(t,x)\qquad q(t,-x)=-q(t,x) \FORALL x\in\mathcal{E}. $$
We define the symmetry space $S$ as follows:
$$S=\begin{Bmatrix}\begin{bmatrix}\m \zeta & q & \langle
q_{\rm i}\rangle & \delta & \eta\m \end{bmatrix}^\intercal\in
\left(L^2(\mathcal{E})\right)^2\times\cline^3\m \quad\text{and}\quad\\
\zeta(-x)=\zeta(x),\quad q(-x)=-q(x) \end{Bmatrix}. $$
To state the result, we introduce the Hilbert space $W$:
$$ W=\begin{Bmatrix} \begin{bmatrix}\m \zeta & q & \langle
q_{\rm i}\rangle & \delta & \eta\m \end{bmatrix}^\intercal\in
\left(\Hscr^1(\mathcal{E})\right)^2\times\cline^3\m\left|\m
\int_{\mathcal{E}}\zeta(x)\dd x+2l\m\delta=0,\right.\vspace{1.9mm}\\
\llbracket q\m\rrbracket=-2l\m\eta,\quad\langle q\rangle=\langle
q_{\rm i}\rangle\quad\text{and}\quad q(-L)=0=q(L') \end{Bmatrix}. $$

\begin{thm}\label{reachspace}
Assume that the object floats in the middle of the fluid domain in
the horizontal direction, i.e. $L'=L$. Then for every
$\tau>\frac{2(L-l)}{\sqrt{gh_0}}$, we have
\begin{equation}\label{resmain}
(W\cap S)\subset\Ran \Phi_\tau\subset (X\cap S),
\end{equation}
where each inclusion is dense and with continuous embedding.
\end{thm}

\begin{rmk}
{\rm  In the symmetric case described above, the average horizontal
discharge $\langle q_{\rm i}\rangle$ and the jump of the elevation
$\llbracket \zeta_{\rm e}\rrbracket$ are both zero, so that the state
$z$ and the linear control system \rfb{lineareqsintro}--\rfb{init1}
can be simplified. We see from the first inclusion in \rfb{resmain}
that any symmetric state with the regularity as in $W$ can be reached
by the control system \rfb{lineareqsintro}--\rfb{init1} from the origin.
The second inclusion in \rfb{resmain} means that the system is not
approximately controllable in $X$, but in its symmetric subspace
$X\cap S$.  More details on this symmetric case are provided in Section
\ref{subsec_symmetric}.
}
\end{rmk}
%%%%%%%%%%++++++++++%%%%%%%%%%++++++++++%%%%%%%%%%++++++++++%%%%%%%%%%
\subsection{Organization of the paper}
In Section \ref{sec_model}, we give a detailed derivation for the full
governing equations of the floating body system in shallow water, in
particular, with a control term. Moreover, we reformulate the equations
into a first-order evolution system only defined in the exterior domain
with tranmission conditions. Then we consider in Section \ref{sec_control}
the linearized system and establish its well-posedness by analysing the 
spectral properties of the evolution operators involved in the control model.
Section \ref{consta_sec} is devoted to studying the reachability
and stabilizability of the linear control system. In the last section, we 
give some comments for the situation in the general case and introduce some 
open problems.

%%%%%%%%%%++++++++++%%%%%%%%%%++++++++++%%%%%%%%%%++++++++++%%%%%%%%%%
\section{Some background on nonlinear modelling of floating body -
	    shallow water interaction}
\label{sec_model}

In this section, we derive the nonlinear governing equations describing
the motion of the floating object in $\Om(t)$, in the presence of a control
applied from the bottom of the object. We follow the approach developed in
\cite{lannes2017dynamics, beck2021freely} with modifications to include the
external force $u(t)$ and the presence of the vertical boundaries of the
water tank. Here we assume that the fluid fills the domain $\Om(t)$, that
it is homogeneous, incompressible, inviscid and irrotational. We also assume
that we are here in a configuration where wave motion is correctly described 
by the nonlinear shallow water equations. We know from
\cite{lannes2017dynamics} that the {\em nonlinear shallow water equations 
with a floating structure} are given, for every 
$t\geq0$ and $x\in\rline$, by
\begin{equation}\label{nsw}
  \left\{ \begin{aligned}
  &\prt_t\zeta+\prt_xq=0,\\
  &\prt_tq+\prt_x\left(\frac{q^2}{h}\right)+gh\m\prt_x\zeta=
  -\frac{h}{\rho}\m\prt_x\underline{P},
  \end{aligned} \right.
\end{equation}
where $g$ is the gravity acceleration. The exterior surface pressure is 
zero, i.e.
$\underline{P}_{\rm e}=0$,
%\begin{equation*}
 % \underline{P}_{\rm e}=0,
%\end{equation*}
while the interior pressure $\underline{P}_{\rm i}$
is determined by the motion of the fluid below the object and also the
control signal $u$. We denote by $\zeta_{\rm w}(t,x)$ the parameterization
of the part of the bottom of the object in contact with the fluid (the
subscript ''$\rm w$'' represents the ''wetted'' part of the object).
Therefore, we have the water surface in the interior domain $\mathcal{I}$
that matches the bottom of the object, i.e.
\begin{equation}\label{match}
  \zeta_{\rm w}(t,x)=\zeta_{\rm i}(t,x)\FORALL x\in\mathcal{I}.
\end{equation}
It is not difficult to see that we have the relation
\begin{equation}\label{relationzeta}
\zeta_{\rm w}(t,x)=\delta(t)+h_{\rm eq}(x)-h_0 \FORALL x\in\mathcal{I}.
\end{equation}
Moreover, we obtain from \rfb{relationzeta} that $\prt_t\zeta_{\rm w}
=\dot{\delta}$, which is the kinematic condition on the water surface.
We consider in what follows restricting the model \rfb{nsw} to the
interval $[-L,L']$, $-L$ and $L'$ being the horizontal coordinates of
the water tank $\Om(t)$, in particular with the control term. To this end,
we observe that the following conditions need to be satisfied:
\begin{itemize}
\item {\em The conservation of the volume of the water}.\\
We first notice that the two vertical boundary of $\Om(t)$ are impermeable,
which implies that
  \begin{equation*}\label{qL1}
    q(t,-L)=0=q(t,L').
  \end{equation*}
Therefore, the conservation of the volume of the water implies that
%$$\prt_t\int_{\mathcal{E}\cup\mathcal{I}}\zeta(t,x)\dd x=0,$$
%which implies by a simple calculation that
  \begin{equation}\label{qboundary}
    q_{\rm i}(t,\pm l)=q_{\rm e}(t,\pm l).
  \end{equation}
\item {\em The conservation of the total energy of the fluid-structure
	  system}.\\
We denote by $E_f$ and $E_s$ the mechanical energy of the fluid and the
mechanical energy of the solid, respectively. Because of the existence
of the vertical force $u$, the total energy of the floating object system
$E_{\rm tot}(t)=E_f(t)+E_s(t)$ should satisfy
  \begin{equation}\label{energyconserve}
    \frac{\dd}{\dd t}E_{\rm tot}(t)=u(t)\dot{\delta}(t).
  \end{equation}
\end{itemize}

%%%%%%%%%%++++++++++%%%%%%%%%%++++++++++%%%%%%%%%%++++++++++%%%%%%%%%%
%\subsection{Governing equations of the floating body system with
%	control}\label{subsec_system}

Based on the conservation of the energy \rfb{energyconserve}, we derive
the boundary conditions of the surface pressure $\underline{P}_{\rm i}$
at the two contact points $x=\pm l$. To do this, we first note that the
mechanical energy of the solid $E_s$ is
$$E_s(t)=mg\delta(t)+\frac{1}{2}m\dot{\delta}^2(t).$$
Recalling the definition of the horizontal discharge $q$, the mechanical
energy of the fluid $E_f$ is
$$E_f(t)=\frac{\rho}{2}\int_{\mathcal{E}\cup\mathcal{I}}\left(
  g\m\zeta^2(t,x)+\frac{q^2}{h}(t,x)\right)\dd x.$$
Note that the object at equilibrium satisfies {\em Archimedes' principle},
we have
\begin{equation}\label{archimede}
	m=\rho\int_{-l}^l\left(h_0-h_{\rm eq}(x)\right)\dd x.
\end{equation}
Newton's law for the motion of the object, together with \rfb{archimede},
implies that
\begin{equation}\label{newtonclear}
	m\ddot{\delta}(t)+2l\rho g\m\delta(t)=\int_{-l}^{l}\left(\underline{P}
	_{\rm i}+\rho g\m\zeta_{\rm i}\right)\dd x+u(t),
\end{equation}
which means that the motion of the object is determined by its weight, the
hydrodynamic force and the external force.
%\begin{equation}\label{newton}
%m\ddot{\delta}(t)=\int_{-l}^{l}\underline{P}_{\rm i}(t,x)\dd x-mg+u(t),
%\end{equation}
Based on the above analysis, we give in the following proposition the
boundary condition of $\underline P_{\rm i}$.

%\begin{rmk}
%{\rm Actually, the energy density $\mathfrak{E}$ of the system \rfb{nsw}
%is defined as
%$$\mathfrak{E}(\zeta, q)=\frac{\rho}{2}\left( g\m\zeta^2(t,x)+
%  \frac{q^2}{h}(t,x)\right).$$
%According to the structure of the system \rfb{nsw}, we have
%\begin{equation}\label{densityflux}
 % \prt_t\mathfrak{E}+\prt_x\mathfrak{F}=\underline{P}\prt_xq.
%\end{equation}}
%\end{rmk}

\begin{prop}
Assume that the functions $\zeta$, $q$, $h$, $\delta$ and $\underline{P}$
are smooth on $\mathcal{I}$ and $\mathcal{E}$. Then the total energy of the
floating body system $E_{\rm tot}$ is conserved if the interior pressure
$\underline P_{\rm i}$ satisfies
\begin{equation}\label{Pbc}
  \underline{P}_{\rm i}(t,\pm l)=\rho g\left(\zeta_{\rm e}(t,\pm l)-
  \zeta_{\rm i}(t,\pm l)\right)+\mathfrak{B}_{\rm e}(t,\pm l)-
  \mathfrak{B}_{\rm i}(t, \pm l),
\end{equation}
where $\mathfrak{B}$ is defined as
\begin{equation}\label{F}
 \mathfrak{B}= \frac{\rho\m q^2}{ 2\m h^2}.
\end{equation}
\end{prop}

\begin{proof}
Taking the derivative of $E_s$ and $E_f$ and using \rfb{newtonclear}, it
is not difficult to obtain that
$$\frac{\dd}{\dd t}E_{\rm tot}(t)=\llbracket\mathfrak{F}_{\rm e}-
\mathfrak{F}_{\rm i} \rrbracket+u(t)\dot{\delta}(t), $$	
where the energy flux $\mathfrak{F}$ is
$$\mathfrak{F}(\zeta,q)=q\left(\rho g\m\zeta+\underline{P}
+\mathfrak{B}\right).$$
Therefore, we conclude that the total energy is conserved in the sense
of \rfb{energyconserve} if $\llbracket\mathfrak{F}_{\rm e}\rrbracket=
\llbracket\mathfrak{F}_{\rm i}\rrbracket$, which follows from the boundary
values of the interior pressure given in \rfb{Pbc}.
%We first note that the variation of the mechanical energy of the solid
%$E_s$ satisfies
%$$\frac{\dd}{\dd t}E_s(t)=\dot{\delta}(t)\int_{-l}^{l}\underline{P}_
 % {\rm i}(t,x)\dd x+u(t)\dot{\delta}(t), $$
%where we used the fact \rfb{newton}. For the mechanical energy of the
%fluid $E_f$, using the relation \rfb{densityflux} we have
%$$\frac{\dd}{\dd t}E_f(t)=\int_{\mathcal{E}\cup\mathcal{I}}\left(-
%  \prt_x\mathfrak{F}+\underline{P}\prt_x q\right)\m\dd x=\llbracket
%  \mathfrak{F}_{\rm e}-\mathfrak{F}_{\rm i} \rrbracket-\dot{\delta}(t)
%  \int_{-l}^l\underline{P}_{\rm i}(t,x)\dd x. $$
%In the above calculation, we used the boundary conditions \rfb{qL}, which
%implies that $\mathfrak{F}(t,-L)=0=\mathfrak{F}(t,L')$. It follows that
%$$\frac{\dd }{\dd t}E_{\rm tot}(t)=\frac{\dd }{\dd t}(E_f(t)+E_s(t))=
%  \llbracket\mathfrak{F}_{\rm e}-\mathfrak{F}_{\rm i}\rrbracket+
%  u(t)\dot{\delta}(t). $$
%Moreover, combined with the continuity of the horizontal discharge
%\rfb{qboundary}, the boundary equations \rfb{Pbc} are equivalent to
%$\llbracket\mathfrak{F}_{\rm e}\rrbracket=\llbracket\mathfrak{F}_{\rm i}
%\rrbracket$. Therefore, if $\Pi_{\rm i}$ satisfies \rfb{Pbc}, the
%mechanical energy of the fluid-structure system is conserved.
\end{proof}

It is worthwhile noting that actually $\mathfrak{B}_{\rm i}(t,\pm l)$
is fully determined by $\delta$ and $\langle q_{\rm i}\m\rangle$. Indeed,
we denote by $h_{\rm w}(t,x)$ the height of the water column
in the interior domain $\mathcal{I}$. By the definition of $h$ and
\rfb{relationzeta} we know that
\begin{equation}\label{hw}
  h_{\rm w}(t,x)=h_0+\zeta_{\rm w}(t,x)=h_{\rm eq}(x)+\delta(t)
  \FORALL x\in\mathcal{I}.
\end{equation}
Together with \rfb{match} and the kinematic condition $\prt_t\zeta_
{\rm w}=\dot{\delta}$, we obtain that the system \rfb{nsw} restricted to
the interior domain, for all $t\geq0$ and $x\in\mathcal{I}$, reads
\begin{equation}\label{intereq}
  \left\{ \begin{aligned}
  &\prt_x q_{\rm i}=-\dot{\delta},\\
  &\prt_tq_{\rm i}+\prt_x\left(\frac{q_{\rm i}^2}{h_{\rm w}}\right)+
  gh_{\rm w}\m\prt_x\zeta_{\rm w}=-\frac{h_{\rm w}}{\rho}\m\prt_x
  \underline{P}_{\rm i}.
\end{aligned} \right.
\end{equation}
The first equation in \rfb{intereq} implies that
\begin{equation}\label{qi}
  q_{\rm i}(t,x)=-x\m\dot{\delta}(t)+\langle q_{\rm i}\rangle
  \FORALL x\in\mathcal{I}.
\end{equation}
Recalling the definition of $\mathfrak{B}$ in \rfb{F} we have
$$\mathfrak{B}_{\rm i}(t,\pm l)=\frac{\rho}{2}\m\left(\frac{
  q_{\rm i}(t,\pm l)}{ \m h_{\rm w}(t,\pm l)}\right)^2=\frac{\rho}{2}
  \left(\frac{\mp l\m\dot{\delta}(t)+\langle q_{\rm i}\rangle}{h_{\rm eq}
  (\pm l)+\delta(t)}\right)^2. $$

Up to now, we obtain the governing equations describing the dynamics
of the floating object in the bounded domain $\Om(t)$ with the
control term $u$. For the sake of convenience, we put all the equations
together as follows, for all $t\geq0$,
\begin{subequations}\label{governing}
  \begin{alignat}{15}
  &\prt_t\zeta+\prt_x q=0 & x\in\mathcal{I}\cup\mathcal{E},
  \label{govern-1}\\
  &\prt_tq+\prt_x\left(\frac{q^2}{h}\right)+gh\m\prt_x\zeta=-\frac{h}
  {\rho}\m\prt_x\underline{P}
  &x\in\mathcal{I}\cup\mathcal{E}, \label{govern-21}\\
  &\underline{P}_{\rm e}(t,x)=0 &x\in\mathcal{E}, \label{govern-3}\\
  & \zeta_{\rm i}(t,x)=\delta(t)+h_{\rm eq}(x)-h_0  &x\in\mathcal{I},
  \label{govern-4}\\
  & \underline{P}_{\rm i}(t,\pm l)=\rho g\left(\zeta_{\rm e}(t,\pm l)
  -\zeta_{\rm i}(t,\pm l)\right)+\mathfrak{B}_{\rm e}(t,\pm l)-
  \mathfrak{B}_{\rm i}(t, \pm l), \label{govern-6}\\
  & m\ddot{\delta}(t)=\int_{-l}^{l}\underline{P}_{\rm i}(t,x)\dd x-
  mg+u(t), \label{govern-7}\\
  & q_{\rm e}(t,-L)=0=q_{\rm e}(t,L'), \qquad q_{\rm i}(t,\pm l)=
  q_{\rm e}(t,\pm l),\label{govern-8}
  \end{alignat}
\end{subequations}
with the given initial data
$$\zeta(0,x)=\zeta_0(x),\quad q(0,x)=q_0(x),\quad \delta(0)=\delta_0,
\quad \dot{\delta}(0)=\delta_1 \FORALL x\in\mathcal{I} \cup\mathcal{E}.$$

\begin{rmk}
{\rm There is another interesting formulation for the governing equations
\rfb{governing}. As in \cite{maity2019analysis}, we can define the
{\em Langrangian} $\Lscr$ and the {\em action functional} $\Sscr$ as
$$\Lscr(\zeta,q,\delta)=(K_f+K_s)-(U_f+U_s), $$
$$\Sscr(\zeta,q,\delta)=\int_0^\tau\left(\Lscr(\zeta,q,\delta)+u\m\delta
  \right)\m\dd t \FORALL
  \tau>0, $$
where $K_f$ and $U_f$ are the kinetic energy and the potential energy
of the fluid, respectively. Similarly, $K_s$ and $U_s$ denote the
corresponding energies for the solid. The equations \rfb{governing}
can be alternatively obtained by using  {\em Hamilton's principle}
(see, for instance, \cite{petit2002dynamics}) with the equations
\rfb{govern-1} and \rfb{govern-4} as constraints.  }
\end{rmk}

%%%%%%%%%%++++++++++%%%%%%%%%%++++++++++%%%%%%%%%%++++++++++%%%%%%%%%%
%\subsubsection{The equations for the interior discharge $q_{\rm i}$}

%%%%%%%%%%++++++++++%%%%%%%%%%++++++++++%%%%%%%%%%++++++++++%%%%%%%%%%
%\subsection{Reformulation of the governing equations}\label{ode}

We next rephrase the governing equations \rfb{governing} as a first-order
evolution system, which will be convenient to study the control problem
described in Section \ref{sec_control}. To do this, we first show that
the pressure term $\underline{P}_{\rm i}$ is actually determined by a
second-order elliptic equation. Based on the formula for $q_{\rm i}$ in
\rfb{qi}, we shall derive the equations for $\delta$ and $\langle q_{\rm
 i}\rangle$ by using the interior equations \rfb{intereq}. We explain in the
following theorem that, for given initial data, the average horizontal
discharge $\langle q_{\rm i}\rangle$ and the displacement $\delta$ are 
totally determined by the quantities in the exterior domain $\mathcal{E}$.
%Here again, in the absence of the external force $u(t)$, the corresponding result
%can be deduced from \cite{beck2021freely} by neglecting the dispersive
%terms.

\begin{thm}\label{transthm}
For smooth solutions, equations \rfb{governing} can be equivalently
rewritten as the following system (involving only the exterior
domain $\mathcal{E}$):
\begin{equation}\label{nswexter}
  \left\{ \begin{aligned}
  &\prt_t\zeta+\prt_xq=0\\
  &\prt_tq+\prt_x\left(\frac{q^2}{h}\right)+gh\m\prt_x\zeta=0
  \end{aligned} \right.
  \qquad\qquad(t\geq 0,\ x\in\mathcal{E}),
\end{equation}
with the transmission conditions
\begin{equation}\label{transcond}
  \langle q\rangle=\langle q_{\rm i}\rangle, \qquad \llbracket q
  \rrbracket =-2l\dot{\delta},
\end{equation}
and the boundary conditions
\begin{equation}\label{qL}
q(t,-L)=0=q(t,L').
\end{equation}
Moreover, the discharge $\langle q_{\rm i}\rangle$ and the displacement
$\delta$ are determined, for every $t\geq0$ and $x\in\mathcal{E}$, by
\begin{equation}\label{newalphadelta}
  \left\{ \begin{aligned}
  &\alpha(\delta)\frac{\dd}{\dd t}\langle q_{\rm i}\rangle+\alpha'(\delta)
  \dot{\delta}\langle q_{\rm i}\rangle=-\frac{1}{2\rho l}\m\llbracket
  \rho g\zeta+\mathfrak{B}\rrbracket,\\
  &M(\delta)\m\ddot{\delta}-2\rho\m l\beta(\delta)\m\dot{\delta}^2+2\rho
  g l\m \delta-\rho\m l\alpha'(\delta)\m\langle q_{\rm i}\rangle^2=2l
  \left\langle\m\rho g\zeta+\mathfrak{B}\right\rangle+u,
  \end{aligned}\right.
\end{equation}
where $\mathfrak{B}$ is introduced \rfb{F} and $\alpha(\delta)$,
$\alpha'(\delta)$, $\beta(\delta)$ and $M(\delta)$ (with $h_{\rm w}$ in
\rfb{hw}) are
\begin{equation}\label{alpha}
\alpha(\delta)=\frac{1}{2l}\int_{-l}^l\frac{1}{h_{\rm w}}\dd x,\qquad
\alpha'(\delta)=-\frac{1}{2l}
\int_{-l}^l\frac{1}{h_{\rm w}^2}\dd x,
\end{equation}
\begin{equation}\label{beta}
M(\delta)=m+\int_{-l}^l\frac{\rho\m x^2}{h_{\rm w}}\dd x,\qquad
\beta(\delta)=\frac{1}{4l}\int_{-l}^l
\frac{x^2}{h_{\rm w}^2}\dd x.
\end{equation}
\end{thm}

\begin{proof}
According to the conservation of the volume \rfb{qboundary} and the
equation for $q_{\rm i}$ in \rfb{qi}, we immediately obtain the
transmission condition \rfb{transcond}. We introduce the hydrodynamic
pressure $\Pi_{\rm i}$ defined by
$$\Pi_{\rm i}:=\underline P_{\rm i}+\rho g\zeta_{\rm i}. $$
Taking the derivative of the second equation in \rfb{intereq} with respect
to $x$ and using the first equation of \rfb{intereq}, we derive that
$\Pi_{\rm i}$ satisfies
\begin{equation}\label{bcp}
\left\{\begin{aligned}
&-\prt_x\left(\frac{h_{\rm w}}{\rho}\prt_x\Pi_{\rm i}\right)=
-\ddot{\delta}+\prt_x^2\left(\frac{q_{\rm i}^2}{h_{\rm w}}\right),\\
&\Pi_{\rm i}(t,\pm l)=\rho g\zeta_{\rm e}(t,\pm l)+\mathfrak{B}_
{\rm e}(t,\pm l)-\mathfrak{B}_{\rm i}(t,\pm l),
\end{aligned}\right.
\end{equation}
where $t\geq0$, $x\in\mathcal{I}$ and $\mathfrak{B}$ is defined in \rfb{F}.
According to what we state around \rfb{qi}, the source term and the
boundary conditions of \rfb{bcp} are determined by $\delta$, $\langle
q_{\rm i}\rangle$ and the exterior functions $\zeta_{\rm e}$, $q_{\rm e}$.

Next we derive the equations for $\delta$ and $\langle q_{\rm i}\rangle$.
Recalling that the function $h_{\rm eq}$ is assumed to be even, we integrate
the second equation of \rfb{intereq} with respect to $x$, which gives
\begin{equation}\label{subqieq}
2l\m\alpha\frac{\dd}{\dd t}\langle q_{\rm i}\rangle+\left\llbracket
\frac{q_{\rm i}^2}{h_{\rm w}^2}\right\rrbracket+\int_{-l}^l
\frac{q_{\rm i}^2}{h_{\rm w}^3}\prt_xh_{\rm w}\m\dd x
=-\frac{1}{\rho}\m\left\llbracket\Pi_{\rm i}\right\rrbracket.
\end{equation}
We further derive from \rfb{subqieq} that
\begin{equation*}\label{qieq}
\alpha(\delta)\frac{\dd}{\dd t}\langle q_{\rm i}\rangle+\alpha'
(\delta)\dot{\delta}\langle q_{\rm i}\rangle=-\frac{1}{2\rho l}
\llbracket \rho g\zeta_{\rm e}+\mathfrak{B}_{\rm e}\rrbracket.
\end{equation*}
In the above calculation, we used the formula for $q_{\rm i}$ in 
\rfb{qi} and the integration by parts. To derive the equation for 
$\delta$, we first obtain from Newton's laws presented in 
\rfb{newtonclear}, by doing an integration by parts, that
\begin{equation}\label{newtonPi}
m\m\ddot{\delta}+2l\rho g\m\delta=2l\left\langle\Pi_{\rm i}\right
\rangle-\int_{-l}^{l}x\m\prt_x\Pi_{\rm i}(t,x)\dd x+u.
\end{equation}
Taking twice integration of the first equation of \rfb{bcp} with respect
to $x$ and using the integration by parts, we obtain the expression for the
second term on the right side of \rfb{newtonPi}. Finally, after doing some
trivial derivation we obtain from \rfb{newtonPi} that
\begin{equation}\label{deltaeq}
M(\delta)\m\ddot{\delta}-2\rho\m l\beta(\delta)\m\dot{\delta}^2+2\rho
g l\m \delta=2l\left\langle\m\rho g\zeta_{\rm e}+\mathfrak{B}_{\rm e}
\right\rangle+\rho\m l\alpha'(\delta)\m\langle q_{\rm i}\rangle^2+u.
\end{equation}
In the above calculation, we used a similar technique as in
\cite{beck2021freely}, so that we omit the details here. For given initial
data of $\zeta$, $q$, $\langle q_{\rm i}\rangle$, $\delta$ and
$\dot{\delta}$, the coupled system \rfb{nswexter}--\rfb{newalphadelta} form
a closed initial boundary value problem.
\end{proof}

It is worthwhile noting that there is an external force $u$ on the right
side of \rfb{deltaeq}, which does not appear in the equations obtained in
\cite{lannes2017dynamics, beck2021freely}. According to Theorem \ref{transthm},
we can further rewrite the system \rfb{nswexter}-\rfb{newalphadelta} into
a first-order evolution system in terms of $\zeta$, $q$, $\langle q_{\rm i}
\rangle$, $\delta$ and $\dot \delta$. This is straightforward and we omit
the details here.

%\begin{cor}\label{odemodel}
%For smooth solutions, the system \rfb{governing} is equivalent to the
%following first-order evolution system, i.e. for every $t\geq0$ and
%$x\in\mathcal{E}$,
%\begin{equation}\label{evolutionnon}
%  \left\{\begin{aligned}
% &\prt_t\zeta=-\prt_x q,\\
 % &\prt_t q=-\prt_x\left(\frac{q^2}{h}\right)-gh\m\prt_x\zeta,\\
  %&\frac{\dd }{\dd t}\begin{bmatrix}\langle q_{\rm i}\rangle
 % \vspace{1.1mm} \\ \delta \\
  %\dot{\delta} \end{bmatrix}=\mathcal{M}^{-1}(\alpha,M)\begin{bmatrix}
 % -\alpha'\dot{\delta}\langle q_{\rm i}\rangle-\frac{1}{2\rho l}
 % \llbracket \rho g\zeta+\mathfrak{B}\rrbracket \vspace{1.3mm} \\
  %\dot{\delta} \\ 2\rho\m l\beta\m\dot{\delta}^2-2\rho g l\m \delta+
  %\rho\m l\alpha'\m\langle q_{\rm i}\rangle^2+2l\left\langle\m\rho
  % g\zeta+\mathfrak{B}\right\rangle+u \end{bmatrix},
  %\end{aligned}\right.
%\end{equation}
%with transmission conditions
%$$\langle q\rangle=\langle q_{\rm i}\rangle, \qquad \llbracket q\m
 % \rrbracket =-2l\dot{\delta},$$
%and boundary conditions
%$$q(t,-L)=0=q(t,L'), $$
%where the matrix $\mathcal{M}(\alpha,M)$ is defined as
%$$\mathcal{M}(\alpha,M)=\begin{bmatrix} \alpha(\delta) & 0 & 0 \\ 0 &
%1 & 0 \\ 0 & 0 & M(\delta)
%\end{bmatrix},$$
%the quantity $\mathfrak{B}$ is introduced in \rfb{F}, $\alpha(\delta)$
%and $\alpha'(\delta)$, $M(\delta)$ and $\beta(\delta)$ are defined
%in \rfb{alpha} and \rfb{beta}, respectively.
%\end{cor}

\begin{rmk}
{\rm The well-posedness theory for \rfb{nswexter}-\rfb{newalphadelta}
is a delicate question, due to the nonlinear couplings: the boundary
conditions \rfb{transcond} of the hyperbolic problem \rfb{nswexter}--\rfb{qL}
require the knowledge of $\langle q_{\rm i}\rangle$ and $\delta$.
Conversely, the equations \rfb{newalphadelta}  require the knowledge
of the trace of $\zeta$ and $q$ at the contact points $x=\pm l$.
An interesting question, which lies outside the scope of the present
work,  is to adapt to our case the local existence theory developed in
\cite{iguchi2021hyperbolic}, which tackles the case of an unbounded
fluid domain and without control.}
\end{rmk}

%%%%%%%%%%++++++++++%%%%%%%%%%++++++++++%%%%%%%%%%++++++++++%%%%%%%%%%++
\section{Well-posedness and spectral analysis of the linearized model}
\label{sec_control}

In this section, we shall work on the linearized version of the first-order
evolution system associated with \rfb{nswexter}-\rfb{newalphadelta}.
Before studying the control problem in Section \ref{consta_sec}, we first
present the linearized model and establish its well-posedness. In the second
part of this section, we focus on the spectral analysis of the semigroup
generator associated to this linearized equation.

Linearizing the system \rfb{nswexter}-\rfb{newalphadelta} in Theorem
\ref{transthm} around the equilibrium state $\begin{bmatrix}\m \zeta &
q & \langle q_{\rm i}\rangle & \delta & \dot \delta\m \end{bmatrix}^\intercal
=\begin{bmatrix}\m 0 & 0 & 0 & 0 & 0\m \end{bmatrix}^\intercal$, we obtain,
for all $t\geq 0$ and $x\in\mathcal{E}$,
\begin{equation}\label{lineareqs}
  \left\{\begin{aligned}
  &\prt_t\zeta=-\prt_x q,\\
  &\prt_t q=-gh_0\m\prt_x\zeta,\\
  &\frac{\dd}{\dd t}\langle q_{\rm i}\rangle=-\frac{g}{2l\underline \alpha}
  \llbracket \zeta\m\rrbracket,\\
  &\ddot{\delta}=-\frac{2\rho g\m l}{\underline M}\delta+\frac{2\rho
  g\m l}{\underline M}\langle\zeta\m\rangle+\frac{1}{\underline M}u,
  \end{aligned}\right.
\end{equation}
with transmission conditions
$$ \langle q\rangle=\langle q_{\rm i}\rangle, \qquad \llbracket q
\rrbracket =-2l\dot{\delta},$$
and boundary conditions
$$q(t,-L)=0=q(t,L'),$$
and the given initial data
$$\zeta(0,x)=\zeta_0(x),\quad q(0,x)=q_0(x),\quad \langle q_{\rm i}
  \rangle(0)=\langle q_{\rm i}\rangle_0, \quad\delta(0)=\delta_0,\quad
  \dot{\delta}(0)=\delta_1.$$
The constants $\underline \alpha$ and $\underline M$ in \rfb{lineareqs}
are
\begin{equation}\label{linearpara}
  \underline \alpha=\alpha(0) \qquad \underline M=M(0),
\end{equation}
where $\alpha(\delta)$ and $M(\delta)$ have been defined in \rfb{alpha}
and \rfb{beta}, respectively.

\subsection{Well-posedness of the linearized system}\label{wellposedl}

Observe that our system has been recast in the exterior domain
$\mathcal{E}$, so we need to rewrite the energy of the whole system in
terms of the exterior functions. Recalling that the mechanical energy
for the fluid and for the object are presented in Section
\ref{sec_model}, we decompose the total energy of the linearized
system \rfb{lineareqs} into the interior part $\underline E_{\rm int}$
and the exterior part $\underline E_{\rm ext}$ as follows:
$$\underline E_{\rm ext}(t)=\frac{\rho}{2}\int_{\mathcal{E}}\left(
  \frac{q^2}{h_0}(t,x)+g\m \zeta^2(t,x)\right)\dd x, $$
$$\underline E_{\rm int}(t)=\frac{\rho}{2}\int_{\mathcal{I}}\left(
  \frac{q^2}{h_{\rm eq}}(t,x)+g \zeta^2(t,x)\right)\dd x+\frac{1}{2}m
  \dot{\delta}^2+mg\delta.$$
The underline in the notation $E$ represents the corresponding energy
for linear system. Using the relation \rfb{relationzeta}, \rfb{hw} and
\rfb{qi}, together with Archimedes' principle \rfb{archimede}, we obtain
that
$$ \underline E_{\rm int}(t)=\frac{1}{2}\m\dot{\delta}^2\left(m+
  \int_{\mathcal{I}} \frac{\rho\m x^2}{h_{\rm eq}}\dd x\right)+\langle
   q_{\rm i}\rangle^2\int_{\mathcal{I}}\frac{\rho}{2h_{\rm eq}}\dd x+
   \m\frac{\rho g}{2}\int_{\mathcal{I}}\left(h_{\rm eq}(x)-h_0\right)^2
   \dd x+\m\rho gl\delta^2. $$
Therefore, we conclude that the total energy for \rfb{lineareqs},
denoted by $\underline E_{\rm tot}$, is
\begin{multline}\label{totalnew}
  \underline E_{\rm tot}(t)=\frac{\rho}{2}\int_{\mathcal{E}}\left(
  \frac{q^2}{h_0}(t,x)+g\m \zeta^2(t,x)\right)\dd x+\frac{1}{2}
  \underline M\dot{\delta}^2+\langle q_{\rm i}\rangle^2\rho\m l\m
  \underline \alpha\\
  +\rho g\m l\m\delta^2+\frac{\rho g}{2}\int_{\mathcal{I}}\left(
  h_{\rm eq}(x)-h_0\right)^2\dd x,
\end{multline}
where $\underline \alpha$ and $\underline M$ are introduced in
\rfb{linearpara}.

Based on the formula of the total energy $\underline E_{\rm tot}$
in \rfb{totalnew}, we introduce the Hilbert space $X$ defined by
\begin{equation}\label{statexn}
  X=\left\{\begin{bmatrix}\m \zeta & q & \langle q_{\rm i}\rangle &
  \delta & \eta\m \end{bmatrix}^\intercal\in \left(L^2(\mathcal{E})
  \right)^2\times\cline^3\m\left|\m\int_{\mathcal{E}}\zeta(x)\dd x+2l
  \m\delta=0\right. \right\},
\end{equation}
endowed with the inner product
\begin{multline}\label{innerxn}
  \left\langle \begin{bmatrix}\m \zeta \\ q \\ \langle q_{\rm i}
  \rangle\vspace{1mm} \\ \delta \\
  \eta\m \end{bmatrix}, \begin{bmatrix}\m \tilde \zeta \\ \tilde q \\
  \langle\tilde q_{\rm i}\rangle\vspace{1mm}\\ \tilde \delta \\ \tilde
  \eta\m \end{bmatrix}\right\rangle_{X}=\frac{\rho g}{2}\langle\zeta,
  \tilde \zeta\rangle_{L^2(\mathcal{E})}+\frac{\rho}{2h_0}\langle q,
  \tilde q\rangle_{L^2(\mathcal{E})}
  +\rho\m l\underline \alpha\langle
  q_{\rm i}\rangle\overline{\langle\tilde q_{\rm i}\rangle}+\rho gl
  \delta\m\overline{\tilde\delta}+\frac{\underline M}{2}\eta\m
  \overline{\tilde \eta}.
\end{multline}

\begin{rmk}
{\rm We can see from \rfb{totalnew} that the total energy only depends
on the functions $\delta$, $\langle q_{\rm i}\rangle$, $\zeta$ and $q$
with the space variable $x\in\mathcal{E}$. The condition
$$\int_{
\mathcal{E}}\zeta(x)\dd x+2l\m\delta=0
$$
in the definition of the space $X$ is motivated by the conservation of
the volume.}
\end{rmk}

Equations \eqref{lineareqs} determine a well-posed linear control system
(also called {\em abstract linear control system} in Weiss
\cite{Weiss1} or Tucsnak and Weiss \cite{Obs_Book}), with state space
$X$ defined in \eqref{statexn} and control space $U=\mathbb{C}$, by choosing
the appropriate spaces and operators. More precisely, let
$A:\Dscr(A)\to X$ and $B\in \Lscr(U,X)$ be defined by
\begin{equation}\label{nonAB}
  A=\begin{bmatrix} 0 & -\frac{\rm d}{{\rm d}x} & 0 & 0 & 0
  \vspace{0.6mm}\\ -gh_0\frac{\rm d}{{\rm d}x}
  & 0 & 0 & 0 & 0 \vspace{1mm}\\ -\frac{g}{2l\underline \alpha}
  \llbracket\cdot\rrbracket & 0 & 0 & 0 & 0 \vspace{0.6mm} \\ 0 & 0
  & 0 & 0 & 1 \vspace{0.6mm} \\ \frac{2}{\underline M}\rho g l\langle
  \cdot\rangle & 0 & 0 & -\frac{2}{\underline M}\rho g l & 0 \m
  \end{bmatrix},\qquad B=\begin{bmatrix} 0 \\ 0 \\ 0 \\ 0
  \vspace{0.5mm}\\\frac{1}{\underline M}\m \end{bmatrix},
\end{equation}
with
\begin{equation}\label{da}
  \Dscr(A)=\begin{Bmatrix} \begin{bmatrix}\m \zeta & q & \langle
  q_{\rm i}\rangle & \delta & \eta\m \end{bmatrix}^\intercal\in
  \left(\Hscr^1(\mathcal{E})\right)^2\times\cline^3\m\left|\m
  \int_{\mathcal{E}}\zeta(x)\dd x+2l\m\delta=0,\right.\vspace{1.9mm}\\
  \llbracket q\m\rrbracket=-2l\m\eta,\quad\langle q\rangle=\langle
  q_{\rm i}\rangle\quad\text{and}\quad q(-L)=0=q(L') \end{Bmatrix}.
\end{equation}
In other words, with the above choice of spaces and operators, the
initial boundary value problem of the system \rfb{lineareqs} can be
rewritten as
\begin{equation}\label{lineareq}
  \left\{ \begin{aligned}
  &\dot{z}=A z+B u,\\
  &z(0)=z_0,
  \end{aligned}\right.
\end{equation}
where $z$ and $z_0$ are
\begin{equation*}\label{znon}
  z=\begin{bmatrix}\m \zeta & q & \langle q_{\rm i}\rangle & \delta
  & \dot{\delta}\m \end{bmatrix}^\intercal,\quad z_0=\begin{bmatrix}
  \m \zeta_0 & q_0 & \langle q_{\rm i}\rangle_0
  & \delta_0 & \delta_1\m \end{bmatrix}^\intercal.
\end{equation*}

The well-posedness of the linearized fluid-structure
system \rfb{lineareq} is a direct consequence of the fact that $B\in
\Lscr(U,X)$ and of following result:

\begin{prop}\label{skewAn}
The operator $A:\Dscr(A)\to X$ defined in \rfb{nonAB}--\rfb{da}
is skew-adjoint. Therefore, it generates a group of unitary operators
on the Hilbert space $X$. Moreover, $A$ has compact resolvents.
\end{prop}

\begin{proof}
We first show that $A$ is skew-symmetric. For the sake of simplicity the
computations leading to the property are performed looking to $X$ as a
Hilbert space over $\mathbb{R}$. For every $z=\begin{bmatrix}
\m \zeta & q & \langle q_{\rm i}\rangle & \delta & \eta\m \end{bmatrix}
^\intercal\in\Dscr(A)$, using the inner product defined in \rfb{innerxn}
we have
$$\left\langle Az, z\right\rangle_{X}=-\frac{\rho g}{2}\left(\left\langle
  \frac{{\rm d}q}{{\rm d}x},\zeta\m\right\rangle_{L^2(\mathcal{E})}+
  \left\langle\frac{{\rm d}\zeta}{{\rm d}x}, q\m
 \right \rangle_{L^2(\mathcal{E})}+\llbracket\zeta_{\rm e}\rrbracket
  \langle q_{\rm i}\rangle-2l\langle\zeta_{\rm e}\rangle\eta\right).$$
By using an integration by parts, we get
$$-\left\langle\frac{{\rm d}q}{{\rm d}x},\zeta\right\rangle_{L^2
 (\mathcal{E})}= \left\llbracket (\zeta q)_{\rm e}\m\right\rrbracket+
  \left\langle\m q,\frac{{\rm d}\zeta}{{\rm d}x}\m\right\rangle_{L^2
 (\mathcal{E})}. $$
Note that the boundary conditions in $\Dscr(A)$ implies that
$$q(l)=\langle q_{\rm i}\rangle-l\eta,\qquad q(-l)=\langle q_{\rm i}
  \rangle+l\eta, $$
which, by a simple calculation, gives that
$$\langle Az,z\rangle_X=0. $$
According to \cite[Section 3.7]{Obs_Book}, we thus obtain that the
operator $A$ is skew-symmetric.

Secondly, we prove that $A$ is onto. For every
$f=\begin{bmatrix}\m f_1 & f_2 & f_3 & f_4 & f_5 \end{bmatrix}^\intercal
\in X$, let us solve the equation
\begin{equation}\label{soln}
  A\begin{bmatrix}\m \zeta \\ q \\ \langle q_{\rm i}\rangle\vspace{1mm}
  \\ \delta \\ \eta\m \end{bmatrix}=\begin{bmatrix}\m -\frac{{\rm d}q}
  {{\rm d}x} \vspace{1mm}\\ -gh_0\frac{{\rm d}\zeta}{{\rm d}x}\vspace{1mm}\\
  -\frac{g}{2l\underline \alpha}\llbracket\zeta_{\rm e}\rrbracket
  \vspace{0.2mm} \\ \eta \\ \frac{2}{\underline M}\rho gl\langle
  \zeta_{\rm e}\rangle-\frac{2}{\underline M}\rho gl\delta
   \end{bmatrix}=\begin{bmatrix}\m f_1 \\ f_2 \\ f_3\\ f_4 \\ f_5
   \end{bmatrix}\quad \text{with}\quad \begin{bmatrix}\m \zeta \\ q \\
   \langle q_{\rm i}\rangle\vspace{1mm} \\ \delta \\ \eta\m \end{bmatrix}
   \in\Dscr(A),
\end{equation}
which immediately implies that $\eta=f_4$. Solving the equation from
the first component of \rfb{soln}, i.e. $-\frac{{\rm d}q}{{\rm d}x}=f_1$
with the boundary conditions $q(-L)=0$ and  $q(L')=0$, we obtain
\begin{equation}\label{qn}
  q(x)=\left\{\begin{aligned}
  &-\int_{-L}^xf_1(\xi)\m\dd \xi \FORALL x\in(-L,-l),\\
  &\int_x^{L'}f_1(\xi)\dd \xi \FORALL x\in(l,L'\m).
  \end{aligned}\right.
\end{equation}
Similarly, from the second equation we get
\begin{equation}\label{zetan}
  \zeta(x)= \left\{\begin{aligned}
  &-\frac{1}{gh_0}\int_{-L}^xf_2(\xi)\m\dd \xi+c_1:=F(x)+c_1
  \FORALL x\in(-L,-l),\\
  &\frac{1}{gh_0}\int_x^{L'}f_2(\xi)\dd \xi+c_2:=G(x)+c_2
  \FORALL x\in(l,L'\m),
  \end{aligned}\right.
\end{equation}
where the constants $c_1$ and $c_2$ are to be determined. The above
formula, together with the last component of \rfb{soln}, gives the
expression for $\delta$:
$$\delta=\frac{1}{2}\left(F(-l)+G(l)+c_1+c_2\right)-\frac{\underline M}
  {2\rho gl}f_5. $$
Moreover, we derive from the third equation of \rfb{soln} that
\begin{equation}\label{cons1}
-\frac{g}{2l\underline \alpha}\left(G(l)-F(-l)+c_2-c_1\right)= f_3.
\end{equation}
Note that the functions $\zeta$ and $\delta$ must satisfy the condition
for the conservation of the volume
$$\int_{\mathcal{E}}\zeta(x)\dd x+2l\delta=0,$$
which implies that
\begin{equation}\label{const2}
  Lc_1+L'c_2=\frac{\underline M}{\rho g}f_5-\int_{-L}^{-l}F(x)\dd x-
  \int_l^{L'}G(x)\dd x-G(l)\m l-F(-l)\m l.
\end{equation}
Combining \rfb{cons1} and \rfb{const2}, we can determine the
constants $c_1$ and $c_2$ in \rfb{zetan}. According to the continuity
of the discharge \rfb{qboundary} and \rfb{qn}, we have $\langle q_{\rm i}
\rangle=\langle q\rangle=\frac{1}{2}(q(l)+q(-l))$. Finally, we still
need to verify that $\llbracket q\m\rrbracket=-2l\eta$. Since $f\in X$,
we have $\int_{\mathcal{E}} f_1(x)\dd x+2\m lf_4=0$, which, together
with \rfb{qn}, implies that $\llbracket q\m\rrbracket=-2lf_4=-2l\eta$.
Thus we have found $z=\begin{bmatrix} \zeta & q & \langle q_{\rm i}
\rangle & \delta &\eta\end{bmatrix}^\intercal\in\Dscr(A)$,
so that \rfb{soln} holds.

According to a classical result \cite[Proposition 3.7.2]{Obs_Book}, we
conclude that $A$ is skew-adjoint and $0\in\rho(A)$. By Stone's theorem
(see, for instance, \cite[Theorem 3.8.6]{Obs_Book}), $A$ generates a
unitary group on $X$. Moreover, it is not difficult to see that $\Dscr(A)$
is compactly embedded in the state space $X$, which implies that the
operator $A$ has compact resolvents.
\end{proof}

Based on Proposition \ref{skewAn}, we denote by $\tline=(\tline_t)_
{t\in\rline}$ the {\em strongly continuous group} (also called $C_0$-group)
generated by the operator $A$. Note that $B\in\Lscr(\cline,X)$,
which is of course an admissible control operator (for this concept,
see, for instance, \cite[Chapter 4]{Obs_Book}). Therefore, $(A,B)$ forms
a well-posed linear control system. According to the
classical semigroup theory, we have the following conclusion.

\begin{thm}\label{wellposedn}
For $u\in L^2_{\rm loc}[0,\infty)$, the initial data $z_0=\begin{bmatrix}
\m \zeta_0 & q_0 & \langle q_{\rm i}\rangle_0 & \delta_0 & \delta_1\m
\end{bmatrix}^\intercal\in X$, the linear system \rfb{lineareq} admits
a unique solution $z$. This solution is given by
$$z(t)=\tline_tz_0+\int_0^t\tline_{t-\sigma}Bu(\sigma)\dd \sigma,$$
and it satisfies
\begin{equation*}\label{reguz}
  z\in C([0,\infty); X).
\end{equation*}
\end{thm}

\subsection{Spectral analysis}\label{controlpro}

In this part, we focus on the study of the spectral structure of the
operator $A$ introduced in \rfb{nonAB}--\rfb{da}. Note that the operator
$A$ is skew-adjoint, the eigenvalues of $A$ are purely imaginary, i.e.
$\sigma(A)\subset{\rm i}\rline$. We give in the following proposition the
characteristic equation for the eigenvalues and the formula for the
corresponding eigenvectors.

\begin{prop}\label{charachA}
For the operator $A$ introduced in \rfb{nonAB}--\rfb{da}, ${\rm i}\m
\omega$ with $\omega\in\rline$ is the eigenvalues of $A$ if and only
if $\omega$ satisfies
\begin{multline}\label{omek}
-\sqrt{\frac{g}{h_0}}\left[2\rho l^2\omega+(\underline M\omega^2-
2\rho g\m l)\frac{1}{l\underline \alpha\m \omega}
\right]\left(f_{\omega}(L)g_{\omega}(L')+f_{\omega}(L')g_{
	\omega}(L)\right)\\
+2(\underline M\omega^2-2\rho g\m l)f_{\omega}(L)f_{\omega}(L')
+\frac{4\rho g\m l}{h_0\underline \alpha}g_{\omega}(L)g_{\omega}(L')=0,
\end{multline}
where $\underline \alpha$ and $\underline M$ are given in \rfb{linearpara},
 $f_{\omega}$ and $g_{\omega}$ are defined as
\begin{equation}\label{fg}
f_{\omega}(x)=\sin\left(\frac{\omega}{\sqrt{gh_0}}(x-l)\right),
\qquad g_{\omega}(x)=\cos\left(\frac{\omega}{\sqrt{gh_0}}(x-l)
\right).
\end{equation}
Moreover, $\phi
=\begin{bmatrix}\m \varphi & \psi & c & a & b\m \end{bmatrix}^\intercal $
is an eigenvector corresponding to the eigenvalue ${\rm i}\m\omega$ if and
only if
\begin{equation}\label{varphik}
\varphi(x)=\left\{\begin{aligned}
&\frac{{\rm i}K_1}{\sqrt{gh_0}}\cos\left(\frac{\omega}{\sqrt{gh_0}}
(L+x)\right)\FORALL x\in(-L,-l),\\
&-\frac{{\rm i}K_2}{\sqrt{gh_0}}\cos\left(\frac{\omega}{\sqrt{gh_0}}
(L'-x)\right)\FORALL x\in(l,L'),
\end{aligned}\right.
\end{equation}
\begin{equation}\label{psik}
\psi(x)=\left\{\begin{aligned}
&K_1\sin\left(\frac{\omega}{\sqrt{gh_0}}(L+x)\right)
\FORALL x\in(-L,-l),\\
&K_2\sin\left(\frac{\omega}{\sqrt{gh_0}}(L'-x)\right)
\FORALL x\in(l,L'),
\end{aligned}\right.
\end{equation}
and
\begin{equation}\label{abck}
\begin{aligned}
c=&\frac{1}{2}(\psi(l)+\psi(-l)),\\
a=\frac{{\rm i}}{2\m\omega l}(\psi(l)-\psi(-l))&,\qquad
b=-\frac{1}{2\m l}(\psi(l)-\psi(-l)),
\end{aligned}
\end{equation}
where $K_1$, $K_2$ are not simultaneously  vanishing real numbers (not
necessarily independent).
\end{prop}

\begin{proof}
%{\bf Step 1}: {\em The structure of the spectrum of $A$.}
Let $\phi=\begin{bmatrix}\m \varphi & \psi & c & a & b\m
\end{bmatrix}^\intercal\in\Dscr(A)$ be the eigenvector of the
operator $A$ corresponding to the eigenvalue ${\rm i}\m\omega$
with $\omega\in\rline$. To obtain the formula of $\phi$, we solve
the equation
\begin{equation}\label{solspectrum}
A\begin{bmatrix}\m \varphi \\ \psi \\ c \\ a \\ b
\end{bmatrix}=\begin{bmatrix}\m -\frac{{\rm d}\psi}{{\rm d}x}
\vspace{1mm}\\ -gh_0\frac{{\rm d}\varphi}{{\rm d}x}\vspace{1mm}\\
-\frac{g}{2l\underline \alpha}\llbracket\varphi \rrbracket
\vspace{1mm}\\ b\\\frac{2}{\underline M}\rho gl \langle\varphi
\rangle-\frac{2}{\underline M}\rho gla\end{bmatrix}={\rm i}\m
\omega\begin{bmatrix}\m\varphi \\ \psi \\ c \\ a \\ b \end{bmatrix},
\end{equation}
where $\underline \alpha$ and $\underline M$ are introduced in
\rfb{linearpara}. Recalling the definition of $\Dscr(A)$ in \rfb{da},
we have
\begin{equation}\label{bc1}
\psi(-L)=0=\psi(L'),\qquad \int_{\mathcal{E}}\varphi(x)\dd x
+2\m l a=0,
\end{equation}
and
\begin{equation}\label{bc2}
\llbracket\psi\m\rrbracket=-2\m l\m b, \qquad \langle\psi\rangle=c.
\end{equation}
Combining the first two equations in \rfb{solspectrum}, we obtain
a second-order differential equation for $\psi$
$$\frac{{\rm d}^2\psi}{{\rm d}x^2}=-\frac{\omega^2}{gh_0}\psi
\FORALL x\in\mathcal{E},$$
which, together with the boundary condition in \rfb{bc1}, implies
that $\psi$ takes the form \rfb{psik}. In \rfb{psik}, $K_1$ and $K_2$
are not simultaneously zero. Notice that $-\frac{{\rm d}\psi}{{\rm d}x}
={\rm i}\m\omega\m\varphi$, we further obtain \rfb{varphik}. Using
the relation between $\varphi$ and $a$ in \rfb{bc1}, we derive that
$$a=\frac{{\rm i}}{2\m\omega l}(\psi(l)-\psi(-l)),$$
which further, by using the fourth equation of \rfb{solspectrum},
implies that
$$b=-\frac{1}{2\m l}(\psi(l)-\psi(-l)).$$
Taking the conditions \rfb{bc2} into account, we have
$$c=\frac{1}{2}(\psi(l)+\psi(-l)).$$
This, together with the third and the last components of \rfb{solspectrum},
imply that the imaginary part of the eigenvalue $\omega$ satisfies
\begin{equation}\label{eaomegak}
\left\{\begin{aligned}
&\frac{{\rm i}\m g}{l\underline \alpha\m\omega}(\varphi(l)-
\varphi(-l))=\psi(l)+\psi(-l),\\
&\frac{\rho g\m l}{\underline M}(\varphi(l)+\varphi(-l))=
\left(\frac{{\rm i}\m\rho g}{\underline
	M\omega}-\frac{{\rm i}\m\omega}{2\m l}\right)(\psi(l)-\psi(-l)),
\end{aligned}\right.
\end{equation}
where $\underline \alpha$ and $\underline M$ are given in \rfb{linearpara}.
Using the formula \rfb{varphik} and \rfb{psik}, the system \rfb{eaomegak}
yields that
\begin{equation}\label{k1k2}
\left[\sqrt{\frac{g}{h_0}}\frac{1}{l\underline\alpha\m\omega}\m
g_{\omega}(L)-f_{\omega}(L)\right]K_1+\left[\sqrt{\frac{g}
	{h_0}}\frac{1}{l\underline \alpha\m\omega}\m g_{\omega}(L')-f_{
	\omega}(L')\right]K_2 =0,
\end{equation}
\begin{multline}\label{omegak}
\left[\sqrt{\frac{g}{h_0}}2\rho\m l^2\m\omega\m g_{\omega}(L)-
(\underline M\omega^2-2\rho g\m l)f_{\omega}(L) \right]K_1\\
+\left[(\underline M\omega^2-2\rho g\m l)f_{\omega}(L')-\sqrt{
	\frac{g}{h_0}}2\rho\m l^2\m\omega\m g_{\omega}(L')\right]K_2=0,
\end{multline}
where $f_\omega$ and $g_\omega$ are introduced in \rfb{fg}.
According to the knowledge of linear algebra, the equations \rfb{k1k2}
and \rfb{omegak} admit non-trivial solutions $\begin{bmatrix}\m K_1 & K_2
\end{bmatrix}^\intercal$, if the
determinant of their coefficient matrix is zero. Therefore,
we obtain the characteristic equation \rfb{omek}.
\end{proof}

Since $A$ is skew-adjoint with compact resolvents (see Proposition
\ref{skewAn}), according to a classical result (see, for instance,
\cite[Chapter 3]{Obs_Book}), we know that $A$ is diagonalizable, also
called Riesz-spectral operator, for instance, in \cite{CZ_THE_BOOK}.
We denote by $(\phi_k)_{k\in\zline^*}$ an orthonormal
basis in $X$ consisting of eigenvectors of $A$ and by $({\rm i}\m
\omega_k)_{k\in\zline^*}$ the corresponding purely imaginary eigenvalues.
Instead of seeing the characteristic equation \rfb{omek}, we observe
that the coefficient matrix of the system \rfb{k1k2}--\rfb{omegak}
can be zero, which implies that the roots of \rfb{omek}, i.e. the
eigenvalues $({\rm i}\m\omega_k)_{k\in\zline^*}$, are not necessarily
simple. We specify this situation in what follows.

\begin{rmk}\label{notsimple}{\rm
Assume that $\kappa:=\underline M-2\rho\m l^3\underline \alpha>0$
and that the parameters $L$, $L'$, $l$ and $h_0$
satisfy
\begin{equation}\label{final1}
\sqrt{\frac{2\rho\m l}{\kappa h_0}}
\frac{L'-L}{\pi}\in\zline,
\end{equation}
and
\begin{equation}\label{condfinal}
\tan\left(\sqrt{\frac{2\rho\m l}{\kappa h_0}}(L-l)\right)=\frac{1}
{l\underline \alpha}\sqrt{\frac{\kappa}{2\rho\m lh_0}}.
\end{equation}
(Recall that the constants $\underline M$ and $\underline \alpha$
have been introduced in \rfb{linearpara}). Then there exist two double
eigenvalues of $A$, denoted by ${\rm i}\m\omega^+$ and ${\rm i}\m
\omega^-$, with
$$\omega^{\pm} =\pm\sqrt{\frac{2\rho gl}{\kappa}}.$$
We are not able to confirm or to inform the existence of  $L,\ L'>0,\
l< \min\{L,L'\}$ and of a function $h_{\rm eq}$ to simultaneously
satisfying the assumptions at the beginning of this remark. However,
it is clear that these conditions are, generically with respect to
the parameters listed above, not satisfied, so that the eigenvalues
are generically simple. Note that if there is at least one double
eigenvalue then the system cannot be controlled (even approximately)
by a scalar input. The result below provides a sufficient condition
in a special case ensuring that all the eigenvalues of $A$ are
simple.
}
\end{rmk}

\begin{prop}\label{heqcond}
Assume that the bottom of the floating object is flat. Let
$h_0> 2\sqrt{\frac{2}{3}}l$ and the function $h_{\rm eq}$ satisfies
\begin{equation}\label{heqrcond}
h_0>h_{\rm eq}\geq \frac{1}{2}\left(h_0+\sqrt{h_0^2-\frac{8}{3}\m l^2}
\right)\quad\text{or}\quad 0<h_{\rm eq}\leq \frac{1}{2}\left(h_0-
\sqrt{h_0^2-\frac{8}{3}\m l^2}\right).
\end{equation}
The all the eigenvalues of $A$ are simple.
\end{prop}

\begin{proof}
Recalling the definition of $h_{\rm eq}$, the flat bottom of the
object implies that $h_{\rm eq}$ is a positive constant function.
Using \rfb{linearpara}, \rfb{alpha} and \rfb{beta}, it is not difficult
to see that if $h_{\rm eq}$ satisfies the condition \rfb{heqrcond} then
$\underline M-2\rho\m l^3\underline \alpha\leq0$. This excludes the
situation of the double eigenvalues discussed in Remark \ref{notsimple}.
\end{proof}

In order to study the reachability and stabilizability properties of
the linearized floating-body system in Section \ref{consta_sec}, it
is necessary to make the inner structure of the eigenvalues clear for
the explicit decay rate of the solution of the control system
\rfb{lineareq}.

\begin{prop}\label{spectrumA}
Assume that the eigenvalues $({\rm i}\m\omega_k)_{k\in\zline^*}$ of
the operator $A$ are simple. Then $(\omega_k)_{k\in\zline^*}$ form a
strictly increasing sequence, i.e. $\displaystyle\lim_{|k|\to\infty}
|\omega_k|=\infty$. Moreover, we assume that $\frac{L'-l}{L-l}$ is a
real algebraic number of degree $n$ with $n\in\nline$ (i.e. it is a
root of a non-zero polynomial of degree $n$ in one variable with
rational coefficients), then there exists $C_0>0$ such that
\begin{equation}\label{gapeigen1}
\inf_{k\in\zline^*}\left|\omega_{k+1}-\omega_k\right|\geq
 C_0\quad \text{if}\m\m\m\m\m\frac{L'-l}{L-l}\in Q\m\m\m\m\m\text{and}
  \m\m\m\m\m\frac{L'-l}{L-l} \neq \frac{r+1}{r}\quad\forall\m\m
  r\in\zline^*,
\end{equation}
\begin{equation}\label{gapeigen2}
\inf_{k\in\zline^*}\left|k\m(\omega_{k+1}-\omega_k)\right|\geq  C_0
\qquad\text{otherwise}.
\end{equation}
\end{prop}

\begin{proof}
%To present the proof clearly, we divide it in two steps.
Since $A$ is skew-adjoint with compact resolvents, according to
\cite[Proposition 3.2.12]{Obs_Book}, the imaginary part of the
eigenvalues $(\omega_k)_{k\in\zline^*}$ can be ordered to form a
strictly increasing sequence such that $\displaystyle\lim_{|k|\to\infty}
|\omega_k|=\infty$. Therefore, it suffices to show that \rfb{gapeigen1}
and \rfb{gapeigen2} holds for $|k|$ large enough. Noting that the
functions $f_{\omega_k}$ and $g_{\omega_k}$ defined in \rfb{fg}
are bounded for large values of $|k|$, we rewrite the equation
\rfb{omek} as
\begin{multline}\label{omekrewrite}
\sqrt{\frac{g}{h_0}}\left(\frac{\underline M}{l\underline
\alpha}+2\rho\m l^2\right)\left(f_{\omega_k}(L)g_{\omega_k}(L')+
f_{\omega_k}(L')g_{\omega_k}(L)\right)\omega_k+r_{\omega_k}\\
=2\underline Mf_{\omega_k}(L)f_{\omega_k}(L')\m\omega_k^2,
\end{multline}
where $r_{\omega_k}$ represents the remaining bounded terms. As
$|k|$ approaches to infinity, we observe that the right hand side
of \rfb{omekrewrite} grows faster than the left side, thus we must
have
$$\lim_{|k|\to\infty}f_{\omega_k}(L)f_{\omega_k}(L')=0. $$
Based on this observation, the eigenvalues of $A$ can be split into
two subsequences $({\rm i}\m\omega_{m_k})_{k\in\zline^*}$ and
$({\rm i}\m\omega_{n_k})_{k\in\zline^*}$, which is induced by
$f_{\omega_k}(L)\to 0$ and $f_{\omega_k}(L')\to0$ as $|k|\to\infty$,
respectively. Therefore, there are two subsequences of $\zline^*$:
$(m_k)_{k\in \mathbb{Z}^*}$ and $(n_k)_{k\in \mathbb{Z}^*}$ such
that, for $|k|$ large enough, we have
\begin{equation*}\label{omwgainf}
\begin{aligned}
\omega_{m_k}=\mu m_k\pi+O(\varepsilon_{m_k}) \qquad &\text{with}
 \m\m\m\lim_{|k|\to\infty}\varepsilon_{m_k}= 0,\\
\omega_{n_k}=\nu n_k\pi+O(\tilde \varepsilon_{n_k})\qquad &
 \text{with}\m\m\m\lim_{|k|\to\infty}\tilde \varepsilon_{n_k}=0,
\end{aligned}
\end{equation*}
where $\mu=\frac{\sqrt{gh_0}}{L-l}$ and $\nu=\frac{\sqrt{gh_0}}{L'-l}$.
For large $|k|$, substituting the first subsequence $(\omega_{m_k})_
{k\in\zline^*}$ into the equation \rfb{omekrewrite}, we have
\begin{multline*}
-\sqrt{\frac{g}{h_0}}\left(\frac{\underline M}{l\underline \alpha}
+2\rho l^2\right)\left[\varepsilon_{m_k}\cos\left(\frac{L'-l}{L-l}
m_k\pi\right)+\sin\left(\frac{L'-l}{L-l}m_k\pi\right)+O(\varepsilon
_{m_k}^2)\right]\omega_{m_k}\\+2\underline M\left[\varepsilon_{m_k}
\sin\left(\frac{L'-l}{L-l}m_k\pi\right)+O(\varepsilon_{m_k}^2)\right]
\omega_{m_k}^2+\text{lower order terms}=0,
\end{multline*}
which implies that $\omega_{m_k}=O(\varepsilon_{m_k}^{-1})$ and thus
we derive that $\varepsilon_{m_k}=O\left({m_k}^{-1}\right)$ for large
$|k|$. Similarly, we also obtain that $\tilde \varepsilon_{n_k}=O
\left({n_k}^{-1}\right)$. Notice that there is a gap between every
two elements both from the sequence $(\omega_{m_k})_{k\in\zline^*}$
or $(\omega_{n_k})_{k\in\zline^*}$. Now we consider the distance
between $(\omega_{m_k})_{k\in\zline^*}$ and $(\omega_{n_k})_{k\in
\zline^*}$. Since the eigenvalues are strictly increasing, we estimate
the difference
\begin{equation}\label{diffomegamn}
\left|\omega_{p+1}-\omega_p\right|=\left|p\nu\pi\left(\frac{\mu}{\nu}-
\frac{p+1}{p}\right)+O\left(\frac{1}{p}\right)\right|,
\end{equation}
where $\omega_p\in(\omega_{m_k})_{k\in\zline^*}$ and $\omega_{p+1}
\in(\omega_{n_k})_{k\in\zline^*}$ correspond to different type of the
eigenvalues. If $\frac{\mu}{\nu}=\frac{L'-l}{L-l}$ is a rational number
but different with $\frac{k+1}{k}$ for any $k\in\zline^*$, we see that
there is a uniform gap between the eigenvalues of $A$. If $\frac{\mu}
{\nu}=\frac{k_0+1}{k_0}$ for some $k_0\in\zline^*$, we obtain from
\rfb{diffomegamn} that the distance between the eigenvalues is of order
$\frac{1}{k}$. If $\frac{\mu}{\nu}$ is not a rational number, then it
is an irrational algebraic number of degree $n\geq2$. According to 
Liouville's approximation theorem (see, for instance, Stolarsky's book
\cite[Chapter 3]{Algebraicapprox}), there exists a constant $C>0$ such
that
$$\left|\frac{\mu}{\nu}-\frac{q}{p}\right|\geq \frac{C}{p^n}, $$
for all rational numbers $\frac{q}{p}$.
Hence, we derive from \rfb{diffomegamn} that $|\omega_{p+1}-\omega_p|
\geq \frac{c}{p}$. Putting all the cases together, we finish the proof.
\end{proof}

\begin{rmk}
{\rm We remark that the set of real algebraic numbers of degree $n$
with $n\in\nline$ contains all rational numbers and some irrational
numbers. All rational numbers form the real algebraic numbers of
degree $1$, and the other part of the real algebraic numbers are
irrational algebraic numbers with $n\geq2$. In particular, the
irrational algebraic numbers of degree $2$ are called {\em quadratic
irrational numbers}.
}
\end{rmk}

\begin{rmk}\label{rem_norm}{\rm
In the proof of Proposition \ref{charachA}, we have obtained the
specific expression for the eigenvectors $\phi_k=\begin{bmatrix}\m
\varphi_k & \psi_k & c_k & a_k & b_k\m \end{bmatrix}^\intercal$, which
is, for every $k\in\zline^*$, given by \rfb{varphik}--\rfb{abck}. Now
we normalize $\phi_k$ in the Hilbert space $X$ introduced in
\rfb{statexn}. By using \rfb{varphik}--\rfb{abck} and after elementary
but tedious calculations, we check that for every $k\in \mathbb{Z}^*$
we have
\begin{multline*}%\label{gammak_def}
  \|\phi_k\|^2_X = \left(\frac{\rho\m l\underline \alpha}{4}+
  \frac{\underline M}{8\m l^2}+\frac{\rho\m g}{4\m\omega_k^2l}\right)
  \left(K_2^2 f_{\omega_k}(L')^2+K_1^2f_{\omega_k}(L)^2\right)\\
  +\left(\frac{\rho\m l\underline \alpha}{2}-\frac{\underline M}{4\m l^2}
  -\frac{\rho\m g}{2\m\omega_k^2l}\right)K_1K_2f_{\omega_k}(L)f_{
  \omega_k}(L')+\frac{\rho}{2h_0}\left(K_1^2(L-l)+K_2^2(L'-l)\right),
\end{multline*}
where $\underline \alpha$, $\underline M$ and $f_{\omega_k}$ are
defined in \rfb{linearpara} and \rfb{fg}, respectively. Therefore, we
obtain the normalized engenvectors $\widehat \phi_k:=\left(\gamma_k
\phi_k\right)_{k\in\zline^*}$ with $\lVert\widehat \phi_k\rVert_X=1$,
where $\gamma_k$ is defined by
\begin{multline*}\label{gammak}
  \gamma_k^{-2} = \left(\frac{\rho\m l\underline \alpha}{4}+\frac{
  \underline M}{8\m l^2}+\frac{\rho\m g}{4\m\omega_k^2l}\right)
  \left(K_2^2 f_{\omega_k}(L')^2+K_1^2f_{\omega_k}(L)^2\right)\\
  +\left(\frac{\rho\m l\underline \alpha}{2}-\frac{\underline M}{4\m l^2}
  -\frac{\rho\m g}{2\m\omega_k^2l}\right)K_1K_2f_{\omega_k}(L)f_{
  \omega_k}(L')\\
  +\frac{\rho}{2h_0}\left(K_1^2(L-l)+K_2^2(L'-l)\right) \qquad\qquad(k\in
  \mathbb{Z}^*).
\end{multline*}
}
\end{rmk}

%%%%%%%%%%++++++++++%%%%%%%%%%++++++++++%%%%%%%%%%++++++++++%%%%%%%%%%+
\section{Reachability and stabilizability of the linearized system}
\label{consta_sec}

\subsection{Some background on controllability and reachable spaces}
We begin by recalling some definitions on the controllability of general
infinite dimensional systems. We consider the abstract differential
equation of the form
\begin{equation}\label{mode}
  \left\{\begin{aligned}
  & \dot z(t) \m=\m Az(t)+Bu(t) \m,\\
  & z(0)=z_0,
  \end{aligned}
\right.
\end{equation}
where $A$ is an infinitesimal generator of a strongly continuous
semigroup $\mathbb{T}=\left(\tline_t\right)_{t\geq 0}$  on a Hilbert
space $X$, and  $B$ is an admissible control operator of the system
\rfb{mode} from the input space $U$ to the state space $X$. This
operator is called {\em bounded} if $B\in\Lscr(U,X)$, which is the
case of interest in this paper. At a given time $t$, the control
$u(t)$ belongs to the input space $U$.

Using the semigroup $\tline$ and the control operator $B$ we can
define the input maps $\left(\Phi_\tau\right)_{\tau\geq 0}$ (already
appearing in \rfb{phitauintro}) by
\begin{equation}\label{phi_tau_new}
\Phi_\tau u=\int_0^\tau \tline_{\tau-s}Bu(s)\dd s  \FORALL \tau>0,
\m\m\m\m u\in  L^2_{\rm loc}([0,\infty); U)).
\end{equation}
An important role in control theory is played by the range of the
operators $\left(\Phi_\tau\right)_{\tau\geq 0}$ defined in
\eqref{phi_tau_new} and denoted, for every $\tau>0$, by $\Ran \Phi_\tau$.
For each $\tau>0$, $\Ran\Phi_\tau$ is called {\em the reachable
space of the system \eqref{mode} in time $\tau$}. These spaces appear,
in particular, in the definition of exact and approximate controllability
which are recalled below (see, for instance, \cite[Chapter 11]{Obs_Book}
or \cite[Chapter 4]{CZ_THE_BOOK}).

\begin{definition}\label{defcontrol}
Let $\tau>0$.
\begin{enumerate}
    \item The system \rfb{mode} is {\em exactly controllable} in time
	$\tau$ if every element of $X$ can be reached from the origin at
	time $\tau$, i.e. if
	$$\Ran \Phi_\tau=X; $$
	\item The system \rfb{mode} is {\em approximately controllable
	in time $\tau$} if
	$$\overline{\Ran \Phi_\tau}=X; $$
\end{enumerate}
\end{definition}

It is well known, see, for instance,  \cite[Chapter 6,8]{Obs_Book},
that approximate controllability can be characterized by duality
as follows:

\begin{prop} \label{prop_char_dual}
Let $\tau>0$. \begin{enumerate}
\item
The system \rfb{mode} is approximately
controllable in time $\tau$ if and only if
\begin{equation*}\label{approxcond}
  B^*\tline_t^*z=0\quad\forall\m\m\m t \in [0,\tau]\m\m\m\m
  \Longrightarrow\m\m\m z=0.
\end{equation*}
\item
Assume that $A$ is skew-adjoint and with compact resolvents, so
that there exists an orthonormal basis $(\phi_k)_{k\in\mathbb{Z}^*}$
in $X$ consisting of eigenvectors of $A$ and let $({\rm i}\m\omega_k)
_{k\in\mathbb{Z}^*}$, with $\omega_k\in \mathbb{R}$ be the 
corresponding eigenvalues. Moreover, assume that the eigenvalues of 
$A$ are simple and that there exists $m,\ \gamma>0$ such that
\begin{equation*}\label{new_gap_con}
 |\omega_k-\omega_l| \geqslant \gamma \qquad\qquad(k,\ l\in
 \mathbb{Z}^*,\ k\neq l,\ |k|\geqslant m,\ |l|\geqslant m).
\end{equation*}
Then the following conditions are equivalent:
\begin{itemize}
\item
The system \rfb{mode} is approximately
controllable in any time $\tau>\frac{2\pi}{\gamma}$;
\item
$B^*\phi_k\neq0$ for every $k\in\mathbb{Z}^*$.
\end{itemize}
\end{enumerate}
\end{prop}

%%%%%%%%%%++++++++++%%%%%%%%%%++++++++++%%%%%%%%%%++++++++++%%%%%%%%%%+
\subsection{Symmetric case}\label{subsec_symmetric}

In this section we come back to the system \eqref{lineareq}, in the
particular case of a symmetric geometry and of initial data satisfying
appropriate symmetry conditions. We show that in this case the state
trajectories of \eqref{lineareq} coincide with those of a ''reduced''
system whose state space is a closed subspace of $X$ defined in
\eqref{statexn} and we study the reachable spaces of this reduced
system.

Let the floating object be in the middle of the fluid domain $\Om$
in the horizontal direction, i.e. $L=L'$, see Figure \ref{fig1f}. We
assume that, at the initial state, the floating body system is at
equilibrium state, i.e. for every $x\in\mathcal{E}$,
$$z_0=\begin{bmatrix}\m \zeta_0 & q_0 & \langle q_{\rm i}\rangle_0
  & \delta_0 & \delta_1\m \end{bmatrix}^\intercal=\begin{bmatrix}
  \m 0 & 0 & 0 & 0 & 0\m \end{bmatrix}^\intercal.$$
In this case, when the object moves in the vertical
direction, the fluid on two sides of the object goes in opposite
directions. To describe this more clearly, we define the Hilbert
space $X_{\rm sym}$ by
%Therefore the linearized system \rfb{lineareq} can be further reduced
%to the one defined only on the right part of the exterior domain, i.e.
%$\mathcal{E}^+=(l,L)$.
\begin{equation}\label{statexr}
X_{\rm sym}=\begin{Bmatrix}  \begin{bmatrix}\m \zeta & q & 0 & \delta &
\eta\m\end{bmatrix}^\intercal\in \left(L^2(\mathcal{E})\right)^2
\times\cline^3\m\left|\m\int_{\mathcal{E}}\zeta(x)\dd x+2l\m\delta=
0\right. \vspace{1.9mm}\\
\zeta(-x)=\zeta(x),\quad q(-x)=-q(x) \end{Bmatrix},
\end{equation}
with the inner product
\begin{equation*}\label{innerxr}
\left\langle \begin{bmatrix}\m \zeta \\ q\vspace{0.6mm}\\ 0\\ \delta \\
\eta\m \end{bmatrix}, \begin{bmatrix}\m \tilde \zeta \\ \tilde q
\vspace{0.6mm}\\ 0 \\ \tilde \delta \\ \tilde \eta\m \end{bmatrix}\right
\rangle_{X_{\rm sym}}= \frac{\rho g}{2}\m\langle\zeta,\tilde \zeta
\rangle_{L^2(\mathcal{E})}+\frac{\rho}{2h_0}\langle q,\tilde q
\rangle_{L^2(\mathcal{E})}+\rho gl\m\delta\m\overline{\tilde \delta}+
\frac{\underline M}{2}\m\eta\m\overline{\tilde \eta},
\end{equation*}
where $\underline M$ has been introduced in \rfb{linearpara}.

\begin{prop}\label{invariant}
The Hilbert space $X_{\rm sym}$ introduced in \rfb{statexr} is
$\tline$-invariant i.e.
$$\tline_t\m z\in X_{\rm sym}\FORALL t\geq0,\m\m\m z\in X_{\rm sym},$$
where $\tline=(\tline_t)_{t\in\rline}$ is the unitary group generated
by the operator $A$ defined in \rfb{nonAB}.
\end{prop}

\begin{proof}
Note that it is suffices to show that the system \rfb{lineareqs} preserves
the symmetry condition in the Hilbert space $X_{\rm sym}$.
Assume that the elevation $\zeta$ and the horizontal discharge $q$ satisfy
\rfb{lineareqs} and have the following properties
\begin{equation}\label{zetaq}
\zeta(t,-x)=\zeta(t,x),\qquad q(t,-x)=-q(t,x) \FORALL t\geq0,\m\m\m
x\in\mathcal{E}.
\end{equation}
We define $\hat \zeta$ and $\hat q$ as
$$ \hat\zeta(t,x)=\zeta(t,-x),\qquad \hat q(t,x)=-q(t,-x)\FORALL t\geq0,
 \m\m\m x\in\mathcal{E},$$
which implies that
$$\langle  q_{\rm i}\rangle=-\langle \hat q_{\rm i}\rangle,\quad
  \llbracket \zeta_{\rm e}\rrbracket=-\llbracket \hat \zeta_{\rm e}
  \rrbracket,\quad \langle\zeta_{\rm e}\rangle=\langle\hat \zeta_{\rm e}
  \rangle.$$
It is not difficult to obtain the corresponding equation for $\hat\zeta$
and $\hat q$, which implies that $\hat\zeta$ and $\hat q$ also satisfy
the system \rfb{lineareqs}.
\end{proof}

Note that the symmetric property \rfb{zetaq} implies
$$\llbracket \zeta_{\rm e}\m\rrbracket=0=\langle q_{\rm e}\rangle=\langle
  q_{\rm i}\rangle\quad \text{and}\quad \langle\zeta_{\rm e}\rangle=
  \zeta_{\rm e}(t,l), $$
which simplify the linear control system \rfb{lineareq}. Since
$X_{\rm sym}$ is a closed subspace of $X$ introduced in \rfb{statexn},
we have the following decomposition
\begin{equation}\label{decomx}
X=X_{\rm sym}\oplus X_{\rm sym}^\perp.
\end{equation}

\begin{rmk}
{\rm The word ''symmetric'' in this section means that not only that
the object is in the center of the domain in the horizontal direction
($L'=L$), but also that the functions $\zeta$ and $q$ satisfy the
symmetry condition \rfb{zetaq}.
	}
\end{rmk}

We thus obtain a new linear system on the spatial domain $\mathcal{E}$.
In this symmetric case, the system \rfb{lineareq} with zero initial
data reduces to the following equations defined on $\mathcal{E}$, i.e.
for all $t\geq0$, $x\in\mathcal{E}$,
\begin{equation}\label{lineareqr}
  \left\{ \begin{aligned}
  &\dot{w}=A_{\rm sym} w+B u,\\
  &w(0)=w_0,
\end{aligned}\right.
\end{equation}
where $w$ and $w_0$ are
\begin{equation*}\label{znonr}
  w=\begin{bmatrix}\m \zeta & q & 0  & \delta & \dot{\delta}\m
  \end{bmatrix}^\intercal,\quad w_0=\begin{bmatrix}\m 0 & 0
  & 0 & 0 & 0\m \end{bmatrix}^\intercal.
\end{equation*}
The operator $A_{\rm sym}:\Dscr(A_{\rm sym})\to X_{\rm sym}$ is densely
defined as
\begin{equation}\label{rAB}
  A_{\rm sym}=\begin{bmatrix} 0 & -\frac{\rm d}{{\rm d}x} & 0& 0 & 0
  \vspace{0.6mm}\\ -gh_0\frac{\rm d}{{\rm d}x}
  & 0 & 0& 0 & 0 \vspace{0.6mm}\\ 0 & 0 & 0&  0 & 0 \\ 0 & 0& 0 & 0 & 1
  \vspace{0.6mm} \\ \frac{2}{\underline M}\rho g l\langle\cdot\rangle
   & 0& 0 &-\frac{2}{\underline M}\rho g l & 0 \m\end{bmatrix},
\end{equation}
with the domain
\begin{equation}\label{dar}
  \Dscr(A_{\rm sym})=\begin{Bmatrix} \begin{bmatrix}\m \zeta & q & 0 &
  \delta & \eta\m \end{bmatrix}^\intercal\in\left(\Hscr^1(\mathcal{E})
  \right)^2\times\cline^3\m\left|\begin{bmatrix}\m \zeta & q & 0& \delta
  & \eta\m \end{bmatrix}^\intercal\in X_{\rm sym},\right.\vspace{1.9mm} \\
  \llbracket q\m\rrbracket=-2l\m\eta,\qquad q(-L)=0=q(L)\end{Bmatrix},
\end{equation}
where $\underline M$ is introduced in \rfb{linearpara}. The control
operator $B$ has been defined in \rfb{nonAB} and we clearly have $B\in\Lscr
(\cline,X_{\rm sym})$.

Note that $A_{\rm sym}$ is the part of $A$ in the closed subspace
$X_{\rm sym}$ of $X$, so it inherits from $A$ the properties of being
skew-adjoint and has compact resolvents. Therefore, it is diagonalizable
and generates a group of unitary operators, denoted by $\tline_{\rm sym}
=(\tline_{{\rm sym}, t})_{t\in\rline}$, on the Hilbert space
$X_{\rm sym}$ defined in \rfb{statexr}. Moreover, according to
\cite[Section 2.4]{Obs_Book}, it is interesting to see from Proposition
\ref{invariant} that $\tline_{\rm sym}$ is the restriction of $\tline$
to $X_{\rm sym}$. Therefore, for $u\in L^2_{\rm loc}[0,\infty)$, the
linear system \rfb{lineareqr} is well-posed and the solution
$w\in C([0,\infty);X_{\rm sym})$.

\begin{rmk}\label{symexplain}
{\rm
Since $B\in\Lscr(\cline, X_{\rm sym})$, it is clear that the input maps
of $(A,B)$ and of $(A_{\rm sym},B)$, the latter being defined by
$$\Phi_{{\rm sym}, \tau}u=\int_0^\tau{\tline_{{\rm sym}, \tau-s}}\m
  Bu(s)\m\dd s\FORALL u\in L^2_{\rm loc}([0,\infty);U)), $$
have the same range, i.e.,  that
$$\Ran \Phi_{\tau}=\Ran \Phi_{{\rm sym}, \tau} \FORALL \tau>0. $$
This means, in particular, that the orthogonal complement space
$X_{\rm sym}^\perp$ in \rfb{decomx} is out of control, justifying
the fact that we concentrate on the reachability of the pair
$(A_{\rm sym},B)$.
}
\end{rmk}

The spectrum of the operator $A_{\rm sym}$ can be obtained, by using
the properties \rfb{zetaq}, from the spectrum of $A$ discussed in
Proposition \ref{charachA}. More precisely, we have:

\begin{prop}\label{sprctrumAr}
Assume that the object is in the middle of the fluid domain which has
the symmetry geometry in the sense \rfb{zetaq}. The eigenvalues of the
operator $A_{\rm sym}$, denoted by $\m{\rm i}\m\omega_{{\rm sym}, k}$,
and the corresponding eigenvectors $\phi_{{\rm sym}, k}=\begin{bmatrix}\m
\varphi_{{\rm sym}, k} & \psi_{{\rm sym}, k} & 0 & a_{{\rm sym}, k} &
b_{{\rm sym}, k}\m\end{bmatrix}^\intercal\in\Dscr(A_{\rm sym})$, for all
$x\in\mathcal{E}$ and $k\in \zline^*$, are
\begin{equation}\label{varphik+}
\varphi_{{\rm sym}, k}(x)=\left\{\begin{aligned}
&\frac{{\rm i}K}{\sqrt{gh_0}}\cos\left(\frac{\omega_{{\rm sym}, k}}
{\sqrt{gh_0}}(L+x)\right)\FORALL x\in(-L,-l),\\
&\frac{{\rm i}K}{\sqrt{gh_0}}\cos\left(\frac{\omega_{{\rm sym}, k}}
{\sqrt{gh_0}}(L-x)\right)\FORALL x\in(l,L),
\end{aligned}\right.
\end{equation}
\begin{equation}\label{psik+}
\psi_{{\rm sym}, k}(x)=\left\{\begin{aligned}
&K\sin\left(\frac{\omega_{{\rm sym}, k}}{\sqrt{gh_0}}(L+x)\right)
\FORALL x\in(-L,-l),\\
&-K\sin\left(\frac{\omega_{{\rm sym}, k}}{\sqrt{gh_0}}(L-x)\right)
\FORALL x\in(l,L),
\end{aligned}\right.
\end{equation}
and
\begin{equation}\label{spectAr}
a_{{\rm sym}, k}=\frac{{\rm i}}{\omega_{{\rm sym}, k}\m l}\m\psi_{{\rm
sym}, k}(l),\quad\m b_{{\rm sym}, k}=-\frac{1}{l}\m\psi_{{\rm sym}, k}(l),
\end{equation}
where $K$ is an arbitrary constant and the imaginary part of the
eigenvalues $\omega_{{\rm sym}, k}$ with $k\in \zline^*$ satisfies
\begin{equation}\label{omegark}
 (\underline M{\omega^2_{{\rm sym}, k}}-2\rho gl)f_{\omega_{{\rm sym},k}}
(L)\\ =\sqrt{\frac{g}{h_0}}2\rho\m l^2\omega_{{\rm sym}, k}\m g_{\omega_
{{\rm sym},k}}(L),
\end{equation}
with $f_{\omega_{{\rm sym}, k}}$ and $g_{\omega_{{\rm sym},k}}$ introduced
in \rfb{fg}. Moreover, the eigenvalues $({\rm i}\m\omega_{{\rm sym}, k})_
{k\in \zline^*}$ are simple and $(\omega_{{\rm sym}, k})_{k\in \zline^*}$
form a strictly increasing sequence, with
$$\displaystyle\lim_{k\in \zline^*,|k|\to\infty}|\omega_{{
\rm sym}, k+1} - \omega_{{\rm sym}, k}|=\frac{\sqrt{gh_0}}{L-l}\pi.$$
\end{prop}

\begin{proof}
Let $\phi_{{\rm sym}}=\begin{bmatrix}\m \varphi_{{\rm sym}} &
\psi_{{\rm sym}} & 0 & a_{{\rm sym}} & b_{{\rm sym}}\m\end{bmatrix}^
\intercal\in\Dscr(A_{\rm sym})$ be an eigenvector of $A_{\rm sym}$
corresponding to the eigenvalue ${\rm i}\m\omega_{\rm sym}$
($\omega_{\rm sym}\in\rline$), we solve the equation
$$A_{\rm sym}\phi_{{\rm sym}}={\rm i}\m\omega_{{\rm sym}}\m\phi_{{
  \rm sym}}.$$
According to Proposition \ref{charachA}, using the symmetry condition
\rfb{zetaq} we obtain that $\phi_{\rm sym}$ take the form
\rfb{varphik+}--\rfb{spectAr}, in particular, the third component of
$\phi_{\rm sym}$ vanishes. In this case, the constants $K_1$ and
$K_2$ in Proposition \ref{charachA} have the relation $K_1=-K_2=K$.
The equation for $\omega_{\rm sym}$ thus becomes
$$\frac{2\rho\m g\m l}{\underline M}\varphi_{\rm sym}(l)={\rm i}\m
\left(\frac{ 2\rho\m g}{\underline M\omega_{\rm sym}}-\frac{\m\omega
_{\rm sym}}{l}\right)\psi_{\rm sym}(l), $$
which gives the characteristic equation \rfb{omegark}. Clearly,
the solutions of \rfb{omegark}, denoted by $(\omega_{{\rm sym},k})
_{k\in\zline^*}$, form a strictly inscreasing sequence. According to
the proof of Proposition \ref{spectrumA}, there is one type of the
eigenvalues in the symmetric case and for large $|k|$
\begin{equation}\label{omegarkasy}
\frac{\omega_{{\rm sym}, k}}{\sqrt{gh_0}}(L-l)=
k\pi+O\left(\frac{1}{k}\right).
\end{equation}
Moreover, \rfb{omegarkasy} implies that there exists $M>0$
such that
\begin{equation}\label{gapsym}
\left|\omega_{{\rm sym}, k+1}-\omega_{{\rm sym}, k}\right|>
\frac{\sqrt{gh_0}}{L-l}\pi \FORALL k\in \zline^*\m\m\m\text{and}
\m\m\m|k|>M,
\end{equation}
which ends the proof.
\end{proof}

\begin{rmk}{\rm
Without using Proposition \ref{spectrumA}, the asymptotic behaviour of
the eigenvalues in \rfb{omegarkasy} can be obtained in an alternative
way. By using the characteristic equation \rfb{omegark}, without loss of
generality, we assume that $\cos\big(\frac{\omega_{{\rm sym}, k}}
{\sqrt{gh_0}}(L-l)\big)$ is non-zero. It follows that
$$\tan\left(\frac{\omega_{{\rm sym}, k}}{\sqrt{gh_0}}(L-l)\right)=
 \sqrt{\frac{g}{h_0}}2\rho\m l^2\frac{\omega_{{\rm sym, }k}}{\underline M{
 \omega_{{\rm sym}, k}}^2-2\rho gl}=O\left(\frac{1}{\omega_{{\rm sym}, k}}
 \right),$$
for large $k\in \zline^*$. Based on the above expression, we assume
that
$$\frac{\omega_{{\rm sym}, k}}{\sqrt{gh_0}}(L-l)=k\pi+\theta_k,$$
with $\theta_k\to0$ as $k\to\infty$. By using the fixed point method
introduced in, for instance, the book \cite[Chapter 7]{gil2007numerical}
or \cite[Lemma A.3]{cindea2015particle}, we derive that $\theta_k=O(k^{-1})$.	
	}
\end{rmk}

By using \rfb{varphik+}--\rfb{spectAr}, we do some trivial calculations
and obtain for every $k\in\zline^*$ that
$$\lVert\phi_{{\rm sym},k}\rVert_{X_{\rm sym}}^2=\left(\frac{\underline M}
{2\m l^2}+\frac{\rho\m g}{{\omega^2_{{\rm sym}, k}}\m l}\right)K^2f^2_{
\omega_{{\rm sym},k}}(L)+\frac{\rho}{h_0}K^2(L-l),$$
where $f_{{\rm sym},k}$ and $\underline M$ are introduced in \rfb{fg} and
\rfb{linearpara} respectively. Now, for every $k\in \zline^*$, we
define $\gamma_{{\rm sym}, k}$ by
\begin{equation}\label{gammark}
  (\gamma_{{\rm sym}, k})^{-2}=\left(\frac{\underline M}{2\m l^2}+
  \frac{\rho\m g}{{\omega^2_{{\rm sym}, k}}\m l}\right)K^2f^2_{\omega_{{
  \rm sym},k}}(L)+\frac{\rho}{h_0}K^2(L-l).
\end{equation}
We therefore obtain the normalized eigenvectors $(\widehat \phi_{{\rm sym},
k})_{k\in\zline^*}:=\left(\gamma_{{\rm sym}, k}\m\phi_{{\rm sym}, k}\right)
_{k\in\zline^*}$ that form an orthonormal basis in $X_{\rm sym}$.

\begin{rmk}\label{rmkomegark}{\rm
As we already realized, the symmetry property \rfb{zetaq} excludes the
case of the double eigenvalues discussed in Section \ref{controlpro}.
Based on the decomposition \rfb{decomx}, we notice that
$(\widehat \phi_{{\rm sym},k})_{k\in\zline^*}$ is a proper subset of
$(\widehat \phi_k)_{k\in\zline^*}$ introduced in Remark \ref{rem_norm}.
Moreover, we have
$$\omega_{{\rm sym},k}=\omega_{j(k)}\FORALL k\in\zline^*,$$
where $\omega_{j(k)}$ is the eigenvalue of $A$ and the subscript
$j(k)\in\zline^*$ can be easily found.
  }
\end{rmk}

%%%%%%%%%%++++++++++%%%%%%%%%%++++++++++%%%%%%%%%%++++++++++%%%%%%%%%%+++
\subsection{Proof of the main result}

The adjoint ${B}^*\in\Lscr(X_{\rm sym},\cline)$ of the control operator
$B$ defined in \rfb{nonAB} is
\begin{equation}\label{Br*}
{B}^*=\begin{bmatrix}\m 0 & 0 &0 & 0 & \frac{1}{2} \end{bmatrix}.
\end{equation}
We are in a position to prove
Theorem \ref{reachspace}.

\begin{proof}[Proof of Theorem \ref{reachspace}]
According to a classical result (see, for instance, \cite[Chapter 4]
{Obs_Book}), we know that for every $\tau>0$ and every $z\in X_{\rm sym}$,
\begin{equation*}\label{phi*}
(\Phi_{{\rm sym}, \tau}^*z)(t)=\left\{\begin{aligned}
&B^*{\tline_{{\rm sym}, \tau-t}^*}\m z\quad \text{for}\m\m \m
t\in[0,\tau],\\
&0\qquad \qquad\qquad\m\m\text{for}\m\m\m\m t>\tau,
\end{aligned}\right.
\end{equation*}
where ${B}^*$ is introduced in \rfb{Br*} and $\tline_{{\rm sym}}$
is the $C_0$-group generated by $A_{\rm sym}$. This implies that for
every $\tau>0$ we have
$$
\left\lVert\left(\Phi^*_{{\rm sym}, \tau}z\right)\right\rVert_{L^2
([0,\tau];U)}^2=\int_0^\tau \left\lVert B^*\tline^*_{{\rm sym}, \tau-t}\m
z\right\rVert_U^2\dd t.
$$
Notice that $0\in\rho(A_{\rm sym})$ and the imaginary part of the
eigenvalues $(\omega_{{\rm sym}, k})_{k\in \zline^*}$ is strictly
increasing, there exists $c>0$ such that $|\omega_{{\rm sym},
k}|\geq c$, which implies that $\gamma_{{\rm sym}, k}$ defined in
\rfb{gammark} is lower bounded by a positive constant. Combining
\rfb{Br*} and Proposition \ref{sprctrumAr}, we have
\begin{equation}\label{lowerbound}
\left|{B}^*(\widehat \phi_{{\rm sym}, k})\right|=\left|\frac{\gamma_{
 {\rm sym}, k}}{2l}\psi_{{\rm sym}, k}(l)\right|\geq C\left|\sin\left(
\frac{\omega_{{\rm sym}, k}}{\sqrt{gh_0}}(L-l)\right)\right|,
\end{equation}
for $k\in \zline^*$. Putting \rfb{omegarkasy} and \rfb{lowerbound}
together, we obtain that
\begin{equation}\label{lowomegark}
\big|{B}^*(\widehat \phi_{{\rm sym}, k})\big|\geq \frac{C}{k}
\FORALL k\in \zline^*.
\end{equation}
Since the operator $A_{\rm sym}$ is diagonalizable and skew-adjoint
on $X_{\rm sym}$, we have
$$\tline_{{\rm sym}, t}\m z=\sum_{k\in \zline^*}{\rm e}^{{\rm i}\m
\omega_{{\rm sym}, k}\m t}\left\langle z, \widehat\phi_{{\rm sym},k}
\right\rangle\widehat \phi_{{\rm sym},k} \FORALL z\in X_{\rm sym}, $$
where $(\widehat \phi_{{\rm sym}, k})_{k\in \zline^*}$, an
orthonormal basis of $X_{\rm sym}$, is introduced around \rfb{gammark}.
Hence, for every $\tau>0$ we have
$$
\int_0^\tau \left\lVert B^*\tline^*_{{\rm sym}, \tau-t}\m
z\right\rVert_U^2\dd t
=\int_0^\tau\bigg|\sum_{k\in \zline^*}{\rm e}^{-{\rm i}\m\omega_
{{\rm sym}, k}t}\langle z,\widehat \phi_{{\rm sym}, k}\rangle B^*
\widehat\phi_{{\rm sym}, k}\bigg|^2\dd t.
$$
Recalling \rfb{gapsym} and using the Ingham theorem (a generalization
of Parseval’s equality, see, for instance, in \cite[Chapter 8]{Obs_Book} or
\cite{KomLorbook}), there exists $\tau_0:=\frac{2(L-l)}{\sqrt{gh_0}}$
such that, for every $\tau>\tau_0$,
\begin{equation}\label{ingsym}
\int_0^\tau\left\lVert B^*\tline_{{\rm sym}, \tau-t}^*\m z\right\rVert
_U^2\m\dd t\geq C\sum_{k\in \zline^*}\big|\langle z,\widehat
\phi_{{\rm sym},k}\rangle\big|^2\big|B^*\widehat \phi_{{\rm sym}, k}
\big|^2.
\end{equation}
Therefore, \rfb{lowomegark} and \rfb{ingsym} imply that, for every
$\tau>\tau_0$,
$$\left\lVert\Phi^*_{{\rm sym}, \tau}z\right\rVert_{L^2([0,\tau];U)}^2
\geq c\left\lVert z\right\rVert_{\Dscr(A_{\rm sym})'}^2\FORALL z\in
X_{\rm sym}, $$
where $\Dscr(A_{\rm sym})'$ is the dual of $\Dscr(A_{\rm sym})$ with
respect to the pivot space $X_{\rm sym}$. Now we introduce the identity
function on $\Dscr(A_{\rm sym})$, denoted by ${\rm id}_{\Dscr(A_{\rm sym})}$,
then of course we have ${\rm id}_{\Dscr(A_{\rm sym})}\in\Lscr(\Dscr(A_{
\rm sym}),X_{\rm sym})$. Note that, for every $\tau>0$, $\Phi_{{\rm sym},
\tau}\in\Lscr(L^2([0, \tau];U);X_{\rm sym})$, we apply next a classical
consequence of the closed graph theorem (see, for instance,
\cite[Proposition 12.1.2]{Obs_Book}), which follows that
$$
\Ran \Phi_{{\rm sym}, \tau}\supset \Dscr(A_{\rm sym}).
$$
Combined with Remark \ref{symexplain}, we conclude that
$\Ran\Phi_\tau\supset \Dscr(A_{\rm sym})$ for every $\tau>\tau_0$.
Recalling that $A_{\rm sym}$ is densely defined, we immediately conclude
that \rfb{resmain} holds.
\end{proof}

%\begin{rmk}{\rm
%The decay rate of the norm of the trajectory of the above closed-loop
%system can be also obtained by using the approach in \cite{chill2019non}.	
%	}
%\end{rmk}

%%%%%%%%%%++++++++++%%%%%%%%%%++++++++++%%%%%%%%%%++++++++++%%%%%%%%%%+++
%\subsection{Non-symmetric case}

%%%%%%%%%%++++++++++%%%%%%%%%%++++++++++%%%%%%%%%%++++++++++%%%%%%%%%%++

\section{Conclusions, comments and open questions}

In this work, we investigate a coupled PDE-ODE system describing the motion
of a  floating body in a free boundary ideal fluid, within the linearized 
shallow water regime. The floating body is constrained to move vertically 
and it is actuated by a control force applied from the bottom of the object. 
Our main result asserts that, provided that, in a symmetric geometrical 
configuration, the system can be steered from rest to any smooth enough 
symmetric wave profile.

We give below, as a consequence of our main theorem, the following result
on the controllability and stabilizability properties of the system 
\rfb{lineareqr}--\rfb{dar}.

\begin{cor}\label{symresult}
Let $L'=L$ and the initial data $\zeta_0$ and $q_0$ satisfy the symmetry 
condition \rfb{zetaq}. Then the linear system defined by 
\rfb{lineareqr}--\rfb{dar} on $X_{\rm sym}$ (briefly designed by
$(A_{\rm sym},B)$), has the following properties
\begin{enumerate}
  \item $(A_{\rm sym},B)$ is not exactly controllable in time $\tau$
  for any finite $\tau>0$;
  \item $(A_{\rm sym},B)$ is approximately controllable on $X_{\rm sym}$ 
  in time $\tau$ for any $\tau>\frac{2(L-l)}{\sqrt{gh_0}}$;
  \item $(A_{\rm sym},B)$ is strongly stabilizable with the feedback
  operator $F=-{B}^*$.
More precisely, there exists $C>0$ such that the closed-loop semigroup
$\tline^{cl}_{\rm sym}$ generated by $A_{\rm sym}-B{B}^*$ satisfies
\begin{equation}\label{retasym}
  \|\tline_{{\rm sym}, t}^{cl} w_0\|_{X_{\rm sym}} \m\leq\m \frac{C}{(1+t)
  ^{\frac12}}\|w_0\|_{\Dscr(A_{\rm sym})} \FORALL w_0\in\Dscr(A_{\rm sym}),
  \ t\geq 0.
\end{equation}
\end{enumerate}
\end{cor}

\begin{proof}
$(1)$ Note that the operator $A_{\rm sym}$ is skew-adjoint and $B\in
\Lscr(\cline,X_{\rm sym})$, then the first assertion follows directly
from Curtain and Zwart \cite[Theorem 4.1.5]{CZ_THE_BOOK} or
\cite[Theorem 5.2.6]{CZ_THE_BOOK} in the same book, since $A_{\rm sym}$
has infinitely many unstable eigenvalues. Equivalently, we know that
the system $(A_{\rm sym},B)$ is not exponentially stabilizable (see,
for instance, Haraux \cite{haraux1989remarque} and Liu \cite{liu1997locally}).
Alternatively, we can apply the main result of Gibson \cite{Gibson}
or Guo, Guo and Zhang \cite[Theorem 3]{guo2007lack}.

$(2)$ The second assertion is a direct consequence of Theorem
\ref{reachspace}. By duality it suffices to show that there exists
$\tau_0>0$, such that for every $\tau>\tau_0$,
\begin{equation}\label{approxr}
  {B}^*\tline_{{\rm sym}, t}^*\m z=0\quad \text{on}\m\m\m
  [0,\tau]\m\m\m\Longrightarrow\m\m\m z=0.
\end{equation}
Let $B^*\tline_{{\rm sym}, t}^*\m z=0$ on $[0,\tau]$ with
$\tau>\frac{2(L-l)}{\sqrt{gh_0}}$, we obtain from \rfb{ingsym} that
$\big\langle z,\widehat \phi_{{\rm sym}, k}\big\rangle=0$ for every
$k\in \zline^*$, which implies that $z=0$. This, together with
Proposition \ref{prop_char_dual}, gives the result.

$(3)$ The approximate controllability of the system $(A_{\rm sym},B)$
is equivalent to the fact that the semigroup  $\tline_{\rm sym}^{cl}$
generated by $A_{\rm sym}-B{B}^*$ is strongly stable (for this, please
refer to Benchimol \cite{benchimol1978note}, Batty and Vu
\cite{batty1990stability}). To obtain the explicit decay rate, we further
conclude from \rfb{lowomegark} and \rfb{ingsym} that
\begin{equation*}\label{dualineq}
\int_0^\tau\left\lVert B^*\tline_{{\rm sym}, t}w_0\right\rVert_U^2\m\dd t
\geq C\m\lVert w_0\rVert_{\Dscr(A_{\rm sym})'}^2 \FORALL w_0\in
\Dscr(A_{\rm sym}).
\end{equation*}
 Hence, we have the interpolation
$$\left[\m\Dscr(A_{\rm sym}), \Dscr(A_{\rm sym})'\m\right]_\theta=
X_{\rm sym}\quad\text{with}\quad \theta=\frac{1}{2}. $$
We apply Theorem 2.4 in \cite{AmTu} and conclude that the semigroup
$\tline_{\rm sym}^{cl}$ generated by $A_{\rm sym}-B{B}^*$ satisfies
\rfb{retasym}.
\end{proof}

The main question left open in our work is the description of the 
reachable space of the considered system without symmetry conditions. 
Using the properties of the eigenvalues of the generator (see Subsection 
\ref{controlpro}) this could be accomplished provided that one has 
lower bounds on $|B^*\widehat \phi_k|$, where $B^*\in\Lscr(X,\cline)$ 
is defined in \rfb{Br*}, and $(\widehat \phi_k)_{k\in\zline^*}$ is the 
orthonormal basis introduced in Remark \ref{rem_norm}. Obtaining such 
lower bounds does not seem an easy task. Indeed, combining 
\rfb{varphik}--\rfb{abck} and \rfb{Br*} we obtain that for every 
$k\in\zline^*$,
\begin{equation}\label{B*est}
|B^*\widehat \phi_k|=\frac{1}{4\m l}\left|\gamma_k\left(K_2f_{
	\omega_k}(L')-K_1f_{\omega_k}(L)\right)\right|,
\end{equation}
where $\widehat \phi_k$ and $\gamma_k$ are introduced in Remark
\ref{rem_norm}, $f_{\omega_k}$ is defined in \rfb{fg}; with constant 
$K_1$ and $K_2$ which we are unable to express in a simple manner in terms 
of $\omega_k$. We also recall from Remark \ref{notsimple} that we are, 
in the general case, unable to confirm or to inform the existence of 
double eigenvalues.

Another open question of interest are the study of the system obtained 
by adding a viscosity term in the shallow water equations, in the spirit 
of Maity et al. \cite{maity2019analysis}. This could lead, in particular, 
to a description of the reachable space for nonlinear systems in which 
the fluid is modeled by the nonlinear shallow water equations. Finally, 
let us mention that an interesting question could be to consider the 
corresponding boundary control problems, in the spirit of 
\cite{Su2020stabilizability} (or a short version \cite{su2021strong}), 
\cite{Su2020asy} and \cite{Su2021capillary}.

%%%%%%%%%%++++++++++%%%%%%%%%%++++++++++%%%%%%%%%%++++++++++%%%%%%%%%%+++
\section*{Acknowledgements}
The authors would like to sincerely thank Prof. David Lannes (from
Universit\'e de Bordeaux) for detailed suggestion on this work.

\bibliography{Su_Tucsnak}

\begin{thebibliography}{10}

\bibitem{AmTu}
{\sc K.~Ammari and M.~Tucsnak}, {\em Stabilization of second order evolution
  equations by a class of unbounded feedbacks}, ESAIM: Control, Optimisation
  and Calculus of Variations, 6 (2001), pp.~361--386.

\bibitem{batty1990stability}
{\sc C.~J. Batty and Q.~P. Vu}, {\em Stability of individual elements under
  one-parameter semigroups}, Transactions of the American Mathematical Society,
  322 (1990), pp.~805--818.

\bibitem{beck2021freely}
{\sc G.~Beck and D.~Lannes}, {\em Freely floating objects on a fluid governed
  by the {B}oussinesq equations}, arXiv preprint arXiv:2102.06947,  (2021).

\bibitem{benchimol1978note}
{\sc C.~D. Benchimol}, {\em A note on weak stabilizability of contraction
  semigroups}, SIAM J. on Control and Optim., 16 (1978), pp.~373--379.

\bibitem{bocchi2020floating}
{\sc E.~Bocchi}, {\em Floating structures in shallow water: local
  well-posedness in the axisymmetric case}, SIAM Journal on Mathematical
  Analysis, 52 (2020), pp.~306--339.

\bibitem{bresch2019waves}
{\sc D.~Bresch, D.~Lannes, and G.~Metivier}, {\em Waves interacting with a
  partially immersed obstacle in the {B}oussinesq regime}, arXiv preprint
  arXiv:1902.04837,  (2019).

\bibitem{cindea2015particle}
{\sc N.~C{\^\i}ndea, S.~Micu, I.~Roven{\c{t}}a, and M.~Tucsnak}, {\em Particle
  supported control of a fluid--particle system}, Journal de Math{\'e}matiques
  Pures et Appliqu{\'e}es, 104 (2015), pp.~311--353.

\bibitem{cretel2011maximisation}
{\sc J.~A. Cretel, G.~Lightbody, G.~P. Thomas, and A.~W. Lewis}, {\em
  Maximisation of energy capture by a wave-energy point absorber using model
  predictive control}, IFAC Proceedings Volumes, 44 (2011), pp.~3714--3721.

\bibitem{CZ_THE_BOOK}
{\sc R.~F. Curtain and H.~Zwart}, {\em An {I}ntroduction to
  {I}nfinite-dimensional {L}inear {S}ystems {T}heory}, Springer Verlag, New
  York, 1995.

\bibitem{Gibson}
{\sc J.~Gibson}, {\em A note on stabilization of infinite dimensional linear
  oscillators by compact feedback}, SIAM J. on Control and Optim., 18 (1980),
  pp.~311--316.

\bibitem{gil2007numerical}
{\sc A.~Gil, J.~Segura, and N.~M. Temme}, {\em Numerical methods for special
  functions}, SIAM, Philadelphia, 2007.

\bibitem{glass2020external}
{\sc O.~Glass, J.~J. Kolumb{\'a}n, and F.~Sueur}, {\em External boundary
  control of the motion of a rigid body immersed in a perfect two-dimensional
  fluid}, Analysis \& PDE, 13 (2020), pp.~651--684.

\bibitem{glass2016motion}
{\sc O.~Glass, C.~Lacave, and F.~Sueur}, {\em On the motion of a small light
  body immersed in a two dimensional incompressible perfect fluid with
  vorticity}, Communications in Mathematical Physics, 341 (2016),
  pp.~1015--1065.

\bibitem{godlewski2018congested}
{\sc E.~Godlewski, M.~Parisot, J.~Sainte-Marie, and F.~Wahl}, {\em Congested
  shallow water model: roof modeling in free surface flow}, ESAIM: Mathematical
  Modelling and Numerical Analysis, 52 (2018), pp.~1679--1707.

\bibitem{guo2007lack}
{\sc F.~Guo, K.~Guo, and C.~Zhang}, {\em Lack of uniformly exponential
  stabilization for isometric {$C_0$-semigroups} under compact perturbation of
  the generators in {B}anach spaces}, Proceedings of the American Math. Soc.,
  135 (2007), pp.~1881--1887.

\bibitem{haraux1989remarque}
{\sc A.~Haraux}, {\em Une remarque sur la stabilisation de certains systemes du
  deuxieme ordre en temps}, Portugaliae {M}athematica, 46 (1989), pp.~245--258.

\bibitem{iguchi2021hyperbolic}
{\sc T.~Iguchi and D.~Lannes}, {\em Hyperbolic free boundary problems and
  applications to wave-structure interactions}, Indiana University Mathematics
  Journal, 70 (2021), pp.~353--364.

\bibitem{john1949motion}
{\sc F.~John}, {\em On the motion of floating bodies {I}}, Communications on
  Pure and Applied Mathematics, 2 (1949), pp.~13--57.

\bibitem{john1950motion}
\leavevmode\vrule height 2pt depth -1.6pt width 23pt, {\em On the motion of
  floating bodies {II}. simple harmonic motions}, Communications on Pure and
  Applied Mathematics, 3 (1950), pp.~45--101.

\bibitem{KomLorbook}
{\sc V.~Komornik and P.~Loreti}, {\em Fourier {S}eries in {C}ontrol {T}heory},
  Monographs in Mathematics, Springer-Verlag, New York, 2005.

\bibitem{lacave2017small}
{\sc C.~Lacave and T.~Takahashi}, {\em Small moving rigid body into a viscous
  incompressible fluid}, Archive for Rational Mechanics and Analysis, 223
  (2017), pp.~1307--1335.

\bibitem{lannes2013water}
{\sc D.~Lannes}, {\em The {W}ater {W}aves {P}roblem: {M}athematical {A}nalysis
  and {A}symptotics}, vol.~188, American Math. Soc., Providence, RI, 2013.

\bibitem{lannes2017dynamics}
\leavevmode\vrule height 2pt depth -1.6pt width 23pt, {\em On the dynamics of
  floating structures}, Annals of PDE, 3 (2017), p.~11.

\bibitem{lannes2020modeling}
\leavevmode\vrule height 2pt depth -1.6pt width 23pt, {\em Modeling shallow
  water waves}, Nonlinearity, 33 (2020), p.~R1.

\bibitem{li2012wave}
{\sc G.~Li, G.~Weiss, M.~Mueller, S.~Townley, and M.~R. Belmont}, {\em Wave
  energy converter control by wave prediction and dynamic programming},
  Renewable Energy, 48 (2012), pp.~392--403.

\bibitem{liu1997locally}
{\sc K.~Liu}, {\em Locally distributed control and damping for the conservative
  systems}, SIAM J. on Control and Optim., 35 (1997), pp.~1574--1590.

\bibitem{maity2019analysis}
{\sc D.~Maity, J.~San~Mart{\'\i}n, T.~Takahashi, and M.~Tucsnak}, {\em Analysis
  of a simplified model of rigid structure floating in a viscous fluid},
  Journal of Nonlinear Science, 29 (2019), pp.~1975--2020.

\bibitem{petit2002dynamics}
{\sc N.~Petit and P.~Rouchon}, {\em Dynamics and solutions to some control
  problems for water-tank systems}, IEEE Transactions on Automatic Control, 47
  (2002), pp.~594--609.

\bibitem{roy2020stabilization}
{\sc A.~Roy and T.~Takahashi}, {\em Stabilization of a rigid body moving in a
  compressible viscous fluid}, Journal of Evolution Equations,  (2020),
  pp.~1--34.

\bibitem{Algebraicapprox}
{\sc K.~B. Stolarsky}, {\em Algebraic numbers and diophantine approximation},
  Marcel Dekker Inc, New York, 1974.

\bibitem{Su2020asy}
{\sc P.~Su}, {\em Asymptotic behaviour of a linearized water waves system in a
  rectangle}, arXiv preprint arXiv:2104.00286,  (2021).

\bibitem{Su2021capillary}
\leavevmode\vrule height 2pt depth -1.6pt width 23pt, {\em Strong stabilization
  of a linearized gravity-capillary water waves system in a tank}, in 2021 IEEE
  60th Annual Conference on Decision and Control (CDC), IEEE, 2021,
  p.~accepted.

\bibitem{Su2020stabilizability}
{\sc P.~Su, M.~Tucsnak, and G.~Weiss}, {\em Stabilizability properties of a
  linearized water waves system}, Systems \& Control Letters, 139 (2020),
  p.~104672.

\bibitem{su2021strong}
\leavevmode\vrule height 2pt depth -1.6pt width 23pt, {\em Strong stabilization
  of small water waves in a pool}, IFAC-PapersOnLine, 54 (2021), pp.~378--383.

\bibitem{Obs_Book}
{\sc M.~Tucsnak and G.~Weiss}, {\em Observation and {C}ontrol for {O}perator
  {S}emigroups}, Birkh\"auser Verlag, Basel, 2009.

\bibitem{Weiss1}
{\sc G.~Weiss}, {\em Admissibility of unbounded control operators}, SIAM J. on
  Control and Optim., 27 (1989), pp.~527--545.

\bibitem{whitham2011linear}
{\sc G.~B. Whitham}, {\em Linear and {N}onlinear {W}aves}, vol.~42, John Wiley
  \& Sons, 2011.

\end{thebibliography}
\bibliographystyle{siam}

\end{document}